\documentclass[a4paper,12pt,reqno]{amsart}
\usepackage{fullpage}
\usepackage{amssymb}
\usepackage{amsmath}
\usepackage{mathtools} 
\usepackage{enumitem}
\usepackage{mathrsfs}
\usepackage[all]{xy}
\usepackage{mathtools}
\usepackage{hyperref}
\usepackage{url}
\usepackage{bbm}
\usepackage{longtable}

\usepackage{lmodern}
\usepackage[OT2, T1]{fontenc}

\DeclareSymbolFont{cyrletters}{OT2}{wncyr}{m}{n}
\usepackage[british]{babel}

\usepackage[margin=1.3in]{geometry}
\numberwithin{equation}{section} \numberwithin{figure}{section}

\DeclareMathOperator{\Coker}{Coker}
\DeclareMathOperator{\Pic}{Pic} \DeclareMathOperator{\Div}{Div}
\DeclareMathOperator{\Gal}{Gal} 
\DeclareMathOperator{\Ker}{Ker}
\DeclareMathOperator{\Aut}{Aut} 
\DeclareMathOperator{\Spec}{Spec}
 
\DeclareMathOperator{\Hom}{Hom} \DeclareMathOperator{\re}{Re}
\DeclareMathOperator{\im}{Im}   
\DeclareMathOperator{\vol}{vol} 

\DeclareMathOperator{\Br}{Br} 

\DeclareMathOperator{\inv}{inv} 
\DeclareMathOperator{\ord}{ord}

\DeclareMathOperator{\Tor}{Tor} \DeclareMathOperator{\Ext}{Ext}

\DeclareMathOperator{\Res}{R} \DeclareMathOperator{\Norm}{N}

\DeclareMathOperator{\Frob}{Frob}
\DeclareMathOperator{\HH}{H}
\DeclareMathOperator{\id}{id}

\let\Im\relax
\DeclareMathOperator{\Im}{Im} 

\DeclareMathSymbol{\Sha}{\mathalpha}{cyrletters}{"58}

\newcommand{\OO}{\mathcal{O}}

\newcommand\exterior{\wedge^2}

\newcommand{\dual}[1]{{#1}^{\wedge}}

\newcommand{\one}{\mathbbm{1}}

\newcommand\FF{\mathbb{F}}

\newcommand\ZZ{\mathbb{Z}}

\newcommand\QQ{\mathbb{Q}}
\newcommand\RR{\mathbb{R}}
\newcommand\CC{\mathbb{C}}
\newcommand\GG{\mathbb{G}}

\newcommand\Gm{\GG_\mathrm{m}}

\newcommand{\Adele}{\mathbf{A}}

\newcommand{\bbF}{{\mathbb F}}

\newcommand{\bbZ}{{\mathbb Z}}

\newcommand{\cA}{{\mathcal A}}

\newcommand{\cO}{{\mathcal O}}

\newcommand{\cX}{{\mathcal X}}

\newcommand{\fp}{\mathfrak{p}}
\newcommand{\gextk}{\ensuremath{G\text{-ext}(k)}}
\newcommand{\pair}[2]{\ensuremath{\langle #1, #2 \rangle}}

\newcommand{\Val}{\Omega_k} 

\newtheorem{lemma}{Lemma}

\newtheorem{theorem}[lemma]{Theorem}
\newtheorem{proposition}[lemma]{Proposition}
\newtheorem{corollary}[lemma]{Corollary}

\theoremstyle{definition}
\newtheorem{example}[lemma]{Example}
\newtheorem{examples}[lemma]{Examples}
\newtheorem{definition}[lemma]{Definition}
\newtheorem{remark}[lemma]{Remark}

\newtheorem*{ack}{Acknowledgements}

\usepackage[usenames,dvipsnames]{color}

\numberwithin{lemma}{section}

\begin{document}

\title{Number fields with prescribed norms}

\author{\sc Christopher Frei}
\address{Christopher Frei\\
TU Graz\\
Institute of Analysis and Number Theory\\
Steyrergasse 30/II\\
8010 Graz\\
Austria.}
\email{frei@math.tugraz.at}
\urladdr{https://www.math.tugraz.at/~frei/}

\author{Daniel Loughran}
  \address{Daniel Loughran \\
	Department of Mathematical Sciences\\
	University of Bath\\
	Claverton Down\\
	Bath\\
	BA2 7AY\\
	UK}
\urladdr{https://sites.google.com/site/danielloughran/}

\author{\sc Rachel Newton \\ With an appendix by Yonatan Harpaz and Olivier Wittenberg}
\address{Rachel Newton\\
Department of Mathematics\\ 
King's College London\\
Strand\\ 
London\\
WC2R 2LS\\
UK.}
   \email{rachel.newton@kcl.ac.uk}
\urladdr{https://racheldominica.wordpress.com/}

\address{Yonatan Harpaz\\
Institut Galil\'ee, Universit\'e Sorbonne Paris Nord, 99~avenue
Jean-Baptiste Cl\'ement, 93430 Villetaneuse, France}
\email{harpaz@math.univ-paris13.fr}
\urladdr{https://www.math.univ-paris13.fr/~harpaz/}

\address{Olivier Wittenberg\\
Institut Galil\'ee, Universit\'e Sorbonne Paris Nord, 99~avenue
Jean-Baptiste Cl\'ement, 93430 Villetaneuse, France}
\email{olivier.wittenberg@math.u-psud.fr}
\urladdr{http://www.math.u-psud.fr/~wittenberg/}

\subjclass[2010]
{11R37 (primary), 
11R45, 
43A70, 
14G05. 
(secondary)}

\begin{abstract}
We study the distribution of extensions of a number field $k$ with fixed abelian Galois group $G$, from which a given finite set of elements of $k$ are norms. In particular, we show the existence of such extensions. Along the way, we show that the Hasse norm principle holds for $100\%$ of $G$-extensions of $k$, when ordered by conductor.
The appendix contains an alternative purely geometric proof of our existence result. 
\end{abstract}

\maketitle

\thispagestyle{empty}

\tableofcontents

\section{Introduction} 

Let $k$ be a number field.
In this paper we are interested in the images of the norm maps $N_{K/k}:K^* \to k^*$ for finite field extensions $K/k$. Specifically, given an element $\alpha \in k^*$ and a finite group $G$, does there exist an extension $K/k$ with Galois group $G$ such that $\alpha$ is a norm from $K$? We are able to answer this question positively if one restricts to abelian extensions of $k$. Furthermore, in the abelian setting, we prove the existence of such an extension from which a given finite set of elements of $k^*$ are norms.

\begin{theorem} \label{thm:existence}
	Let $k$ be a number field, $G$ a finite abelian group and $\mathcal{A} \subset k^*$ a finitely generated subgroup. Then there exists an abelian
	extension $K/k$ with Galois group $G$ such that every element of $\mathcal{A}$ is a norm from $K$.
\end{theorem}
As an application, we obtain the following corollary.

\begin{corollary}
	Let $k$ be a number field, $G$ a finite abelian group and $S$ a finite set of places of $k$. Then there exists an abelian
	extension $K/k$ with Galois group $G$ such that every $S$-unit of $k$ is a norm from $K$.
\end{corollary}

We prove Theorem \ref{thm:existence} by \emph{counting} the collection of abelian extensions under consideration; we obtain an asymptotic formula for the number of such extensions of bounded conductor, and show explicitly that the leading constant in this formula is non-zero. In particular, we prove the existence of infinitely many extensions with the desired properties. The strategy of proving existence via counting is widely used in analytic number theory, for example in the context of the Hardy--Littlewood circle method. Our proof of Theorem \ref{thm:existence} seems to be the first case where it is implemented for number fields. Our methods even allow us to prove existence of such an extension $K/k$ which satisfies any finite collection of admissible local conditions (Corollary \ref{cor:existence}).

Before we can explain these more general results, we must introduce some notation. Fix a choice of algebraic closure $\overline{k}$ of $k$ and let $G$ be a finite abelian group. By a \emph{$G$-extension} of $k$, we mean a surjective continuous homomorphism $\varphi:\Gal(\overline{k}/k)\to G$. This corresponds to choosing an extension $k\subset K \subset \bar{k}$ together with an isomorphism $\Gal(K/k) \cong G$. Keeping track of the isomorphism with $G$ simplifies the set-up and the counting. It has no qualitative effect on the results;
forgetting the choice of isomorphism merely scales all the counting results by $|\Aut(G)|$. 
  We write $G$-\textrm{ext}$(k)$ for the set of all $G$-extensions of $k$. Given $\varphi\in G$-\textrm{ext}$(k)$, we write $K_\varphi$ for the corresponding number field, and $\Phi(\varphi)$ for the norm of the conductor of $K_\varphi$ (viewed as an ideal of $k$). Moreover, we write $\Adele_{K_\varphi}^*$ for the ideles of the number field $K_{\varphi}$. We are interested in the counting functions
\begin{align}
	N(k,G,B) &= \#\{ \varphi \in \gextk : \Phi(\varphi) \leq B\}, \nonumber\\
	N_{\mathrm{loc}}(k,G,\mathcal{A},B) &= \#\{ \varphi \in \gextk : \Phi(\varphi) \leq B, \mathcal{A} \subset \Norm_{K_\varphi/k} \Adele_{K_\varphi}^* \}, \label{def:counting_functions}\\
	N_{\mathrm{glob}}(k,G,\mathcal{A},B) &= \#\{ \varphi \in \gextk : \Phi(\varphi) \leq B, \mathcal{A} \subset \Norm_{K_\varphi/k} K_\varphi^* \}.\nonumber
\end{align}

The first counts all $G$-extensions $\varphi$ of $k$ of bounded conductor, the second counts only those for which every element of $\mathcal{A}$ is everywhere locally a norm, the third only those for which every element of $\mathcal{A}$ is a global norm.

An asymptotic formula for $N(k,G,B)$ was first obtained by Wood in \cite{Woo10}, building on numerous special cases. In this paper we obtain asymptotic formulae for the other counting functions. Our  formulae are stated in terms of the invariant $\varpi(k,G,\mathcal{A})$ which we now define.

\begin{definition}\label{def:varpi_intro}
  Let $k$ be a number field, $G$ a finite abelian group, and $\mathcal{A}\subset k^*$ a finitely generated subgroup. For $d\in\ZZ_{\geq 1}$, let $k_d=k(\mu_d,\sqrt[d]{\mathcal{A}})$. We define
  \begin{equation*}
    \varpi(k,G,\mathcal{A}) = \sum_{g\in G\smallsetminus \{\id_G\}}\frac{1}{[k_{|g|}:k]},
  \end{equation*}
  where $|g|$ denotes the order of $g$ in $G$ and $\id_G \in G$ is the identity element.
\end{definition}

\begin{theorem} \label{thm:global}
	Let $k$ be a number field, $G$ a non-trivial finite abelian group, and $\mathcal{A} \subset k^*$ a finitely generated subgroup.
	Then 
	$$N_{\mathrm{glob}}(k,G,\mathcal{A},B) \sim c_{k,G,\mathcal{A}}B(\log B)^{\varpi(k,G,\mathcal{A})-1}$$
	as $B \to \infty$, for some $c_{k,G,\mathcal{A}} > 0$.
      \end{theorem}

This theorem gives an asymptotic formula for the number of $G$-extensions from which every element of $\mathcal{A}$ is a global norm. It is natural to ask how the number of such extensions compares with the total number $N(k,G,B)$ of $G$-extensions of $k$ of conductor bounded by $B$. We observe that $N(k,G,B)=N_{\mathrm{glob}}(k,G,\{1\},B)$ and note that in this case the formula of Theorem~\ref{thm:global} agrees with \cite[Thm.~3.1]{Woo10}.

\begin{example}
In the special case where $G=\ZZ/2\ZZ$ and $\alpha\in k^*\smallsetminus k^{*2}$, we compute $\varpi(k,\ZZ/2\ZZ,\langle \alpha\rangle) = 1/2$ and thus $N_{\mathrm{glob}}(k,\ZZ/2\ZZ,\langle\alpha\rangle,B) \sim c_{k,\ZZ/2\ZZ,\langle\alpha\rangle}B(\log B)^{-1/2}$. When compared to the asymptotic $N(k,\ZZ/2\ZZ,B) \sim c_{k,\ZZ/2\ZZ} B$, this shows that for $100\%$ of quadratic extensions of $k$ the number $\alpha$ is not a norm.
\end{example}

The next theorem generalises this observation. It says that, unless we are in a very special case, for $100\%$ of $G$-extensions of $k$ not all elements of $\mathcal{A}$ are norms.

\begin{theorem} \label{thm:norm_rare}
	Let $k$ be a number field, $G$ a non-trivial finite abelian group of exponent $e$, and $\mathcal{A} \subset  k^*$ a finitely generated subgroup. Then the following are equivalent:
        \begin{enumerate}
        \item $\lim_{B \to \infty} \frac{N_{\mathrm{glob}}(k,G,\mathcal{A},B)}{N(k,G,B)} > 0$;
        \item $\mathcal{A} \subset k(\mu_d)^{*d}$ for all $d\mid e$;
        \item $\mathcal{A} \subset k_v^{*e}  \text{ for all but finitely many places } v$ of $k$.
        \end{enumerate}
\end{theorem}

There is a nice cohomological way to interpret the condition (3) in Theorem~\ref{thm:norm_rare} via certain Tate--Shafarevich groups (see \S \ref{sec:cor_norm_rare}).
Together with some class field theory, this will allow us to deduce the following result.

\begin{corollary} \label{cor:norm_rare}
	Let $\mathcal{A} \subset k^*$ be a finitely generated subgroup and let $e$ be the exponent of $G$. Then the limit 
	\begin{equation}\label{eqn:norm_rare}
	\lim_{B \to \infty} \frac{N_{\mathrm{glob}}(k,G,\mathcal{A},B)}{N(k,G,B)}
	\end{equation}
		\begin{enumerate} 
	\item[(i)] only depends on the image $\mathcal{A}k^{*e}$ of $\mathcal{A}$ in $k^*/k^{*e}$;
	\item[(ii)] equals one if $\mathcal{A} \subset k^{*e}$; 
	\item[(iii)] is zero
	for all but finitely many finite subgroups
	$\mathcal{A}k^{*e}\subset k^*/k^{*e}$;
	\item[(iv)] 
	 is zero for 
	 all finitely generated subgroups $\mathcal{A} \not \subset k^{*e}$ if and only if the extension $k(\mu_{2^r})/k$ is cyclic,
	 where $2^r$ is the largest power of $2$ dividing $e$.
	\end{enumerate}
\end{corollary}

Condition $(iv)$ holds for example if $8 \nmid e$ or $\mu_{e} \subset k^*$.
Our next result shows that if $G$ is cyclic then in order to have \[0< \lim_{B \to \infty} \frac{N_{\mathrm{glob}}(k,G,\mathcal{A},B)}{N(k,G,B)}<1,\] 
for some choice of $\mathcal{A}$, the field $k$ must have more than one prime lying above $2$.

\begin{theorem}\label{thm:positivenot1}Let $k$ be a number field, let $\mathcal{A} \subset k^*$ be a finitely generated subgroup, and let $G$ be a finite cyclic group. Suppose that $k$ has only one prime lying above $2$. Then the following are equivalent:
\begin{enumerate}
\item
$\lim_{B \to \infty} \frac{N_{\mathrm{glob}}(k,G,\mathcal{A},B)}{N(k,G,B)} >0$;
\item every element of $\mathcal{A}$ is a global norm from every $G$-extension of $k$.
\end{enumerate}
\end{theorem}

A necessary condition for an element of $k$ to be a global norm is that it is a norm everywhere locally.  However, this is not a sufficient condition in general due to possible failures of the \emph{Hasse norm principle} (HNP). Nevertheless, to prove Theorem \ref{thm:global}, we reduce to the case of everywhere local norms via the following theorem, which shows that, when ordered by conductor, ``most'' abelian extensions satisfy the Hasse norm principle.

\begin{theorem} \label{thm:HNP_rare}
	Let $k$ be a number field, $G$ a finite abelian group, and $\mathcal{A} \subset k^*$ a finitely generated subgroup.
	Then
	$$
	\lim_{B \to \infty}
		\frac{\#\{ \varphi \in \gextk : \Phi(\varphi) \leq B, \mathcal{A} \subset \Norm_{K_\varphi/k} \Adele_{K_\varphi}^*, K_\varphi \text{ fails the HNP} \}}{N_{\mathrm{loc}}(k,G,\mathcal{A},B)} =0.
	$$
\end{theorem}


In particular Theorem \ref{thm:HNP_rare} implies that
$$
 	\lim_{B \to \infty} \frac{N_{\mathrm{glob}}(k,G,\mathcal{A},B)}{N_{\mathrm{loc}}(k,G,\mathcal{A},B)} =1.
 	$$
Theorem \ref{thm:global} can thus be proved via  an asymptotic formula for $N_{\mathrm{loc}}(k,G,\mathcal{A},B)$, which we obtain in Theorem \ref{thm:loc}. We prove Theorem \ref{thm:HNP_rare} using a purely local criterion for failure of the Hasse norm principle (Proposition \ref{prop:HNP}).
Taking $\mathcal{A}=\{1\}$ in Theorem~\ref{thm:HNP_rare}, we obtain the following result.
\begin{corollary}\label{cor:100HNP}
	Let $k$ be a number field and $G$ a finite abelian group. Then $100\%$ of $G$-extensions of $k$, ordered by conductor, satisfy the Hasse norm principle.
\end{corollary}

Corollary~\ref{cor:100HNP} stands in stark contrast to the results of \cite{HNP}, where a dichotomy occurs when counting by discriminant: in \emph{op.~cit.}~we showed that for certain finite abelian groups $G$ a positive proportion of $G$-extensions can fail the Hasse norm principle, when ordered by discriminant. This contrasting behaviour illustrates the fact, already observed by Wood in \cite{Woo10}, that counting by conductor often leads to more natural statements than counting by discriminant. In fact, after seeing the results we obtained in \cite{HNP} when counting extensions ordered by discriminant, Wood remarked that the dichotomy we had observed should disappear when ordering by conductor, and conjectured the statement of Corollary~\ref{cor:100HNP}. 

There are two reasons why it seems quite difficult to prove Theorem \ref{thm:existence} when counting by discriminant, rather than conductor. Firstly, the condition that every element of $\mathcal{A}$ is a norm everywhere locally may be only rarely satisfied and, in the setting of  \cite[Thm.~1.4]{HNP} where a positive proportion of $G$-extensions fail the Hasse norm principle, it becomes challenging to show the existence of a $G$-extension for which every element of $\mathcal{A}$ is a norm everywhere locally and the Hasse norm principle holds. Secondly, the leading constant obtained when counting by discriminant is very complicated, with potential for further cancellation, so it is difficult to prove its positivity, whereas when counting by conductor we have a simple criterion for positivity of the leading constant (see Theorem~\ref{thm:main}).

The counting techniques employed in this paper are fairly robust and enable us to prove a strengthening of Theorem \ref{thm:existence} in which we impose local conditions at finitely many places. See Theorem \ref{thm:main} and Corollary \ref{cor:existence_S} for precise statements.

Our work on the statistical behaviour of the Hasse norm principle brings together two major areas of modern number theory: namely, counting within families of number fields, and the quantitative study of the failure of local-global principles. Notable recent papers on the statistics of number fields include \cite{ASVW17}, \cite{BL18}, \cite{BST13}, \cite{BV15}, \cite{EPW17}, \cite{FW17}, \cite{HSV18}, \cite{PTBW17} and \cite{Woo17}. Some significant contributions to the  study of local-global principles in families include \cite{Bha14c},  \cite{BB14b}, \cite{BB14a}, \cite{BBL}, \cite{EJ15}, \cite{LS16} and \cite{MacedoAn}. For a summary of recent progress on counting failures of the Hasse principle, see \cite{BrowningSurvey}. More specifically, the statistical behaviour of the Hasse norm principle is examined in \cite{BN15}, \cite{Macedo18} and \cite{Rome17}. In particular, in \cite{Rome17} Rome obtains an asymptotic formula for the number of biquadratic extensions of $\QQ$ (ordered by discriminant) which fail the Hasse norm principle. Obtaining asymptotic formulae for the number of such failures for other classes of field extensions would seem to be an interesting problem.

Below, we give some examples illustrating our results in a variety of settings to demonstrate the wide range of phenomena manifested by norms in extensions of number fields.
\begin{examples} \label{examples} \hfill
\begin{enumerate}
	\item Take $G = \ZZ/n\ZZ$ with $8 \nmid n$ and $\alpha \in k^*$ not an $n$th power.
	Then Corollary~\ref{cor:norm_rare} implies that for $100\%$ of all $\ZZ/n\ZZ$-extensions of $k$ ordered by conductor, $\alpha$ is not a norm.
	In the special case $n = 2$ of quadratic extensions, this result can be 
	proved using standard techniques in analytic number theory;
	all other cases are new.
	
	\item Take $k = \QQ, \alpha = 16$ and $G = \ZZ/8\ZZ$.
	As is well known, $16$ is an $8$th power in $\QQ_p^*$ for all odd primes $p$ and in $\RR^*$. It therefore follows from Theorems~\ref{thm:norm_rare} and~\ref{thm:positivenot1} that
	$16$ is a norm from every $\ZZ/8\ZZ$-extension $K/\QQ$, despite not being an $8$th power in $\QQ$.

	\item Take $k=\mathbb{Q}(\sqrt{17})$,  $\alpha = 16$ and $G = \ZZ/8\ZZ$. Then, as above, we see that $16$ is locally an $8$th power at all places $v$ such that $v\nmid 2$. Hence $16$ is a local norm from all $\ZZ/8\ZZ$-extensions of $k$ at all places $v\nmid 2$. However, let $\mathfrak{p}, \mathfrak{q}$ be the two primes of $k$ above $2$.  By \cite[Thm.~9.2.8]{NSW08} there exists a $\mathbb{Z}/8\ZZ$-extension $F/k$ such that $F_\mathfrak{p}/k_\mathfrak{p}$ is unramified of degree $8$. Therefore, $16$ is not a local norm from $F_\mathfrak{p}/k_\mathfrak{p}$,
	and consequently not a global norm from $F/k$. Given the existence of one such an extension, an application of \cite[Cor.~1.7]{Woo10} (or Theorem \ref{thm:main}) yields the existence of a positive proportion of $\ZZ/8\ZZ$-extensions $K/k$ which are unramified of degree $8$ over $\fp$, thus the limit \eqref{eqn:norm_rare} is positive but not equal to $1$ in this case.

	Let us explain in more detail why \cite[Thm.~9.2.8]{NSW08} applies here but not in the previous example. Recall that a place $v$ of a number field $L$ is said to \emph{split} (or \emph{decompose}) in an extension $M/L$ if there exist at least two distinct places of $M$ above $v$. All places of $\QQ$ apart from $2$ split in the non-cyclic extension $\QQ(\mu_8)/\QQ$, so that $(\QQ,8,\Omega_\QQ\setminus\{2\})$ is a so-called special case and \cite[Thm.~9.2.8]{NSW08} does not apply in example (2). However, in example (3), $\mathfrak{q}$ is non-split in $k(\mu_8)/k$: both $\mathfrak{p}$ and $\mathfrak{q}$ are totally ramified in $k(\mu_8)/k$, since $2$ is split in $k/\QQ$ and totally ramified in $\QQ(\mu_8)/\QQ$. Therefore, $(k,8,\Omega_k\setminus\{\mathfrak{p}\})$ is not a special case and \cite[Thm.~9.2.8]{NSW08} can be applied in example (3).

	\item Take $k=\QQ, \alpha = 5^2$ and $G=(\ZZ/2\ZZ)^2$. A simple argument
	(cf.~Lemma~\ref{lem:Rachel}) shows that $5^2$ is a norm everywhere locally
	from \emph{every} biquadratic extension of $\QQ$. By Theorem \ref{thm:HNP_rare},
	it is thus a global norm from $100\%$ of biquadratic extensions of $\QQ$ ordered by conductor.
	However, $5^2$ is not a global norm from $K=\QQ(\sqrt{13},\sqrt{17})$ (failure of the Hasse norm principle \cite[p.~360, Exercise 5.3]{CF}). Therefore, it is not true that $5^2$ is a global
	norm from \emph{every} biquadratic extension of $\QQ$.
		\end{enumerate}
	
\end{examples}

\begin{remark}
A simple application of local class field theory (Lemma~\ref{lem:Rachel}) shows that every element of $k^{*e}$ is everywhere locally a norm from every $G$-extension of $k$, where $e$ denotes the exponent of $G$. Using this, one can show that in our results, the assumption that $\mathcal{A}$ is a finitely generated subgroup of $k^*$ can be replaced by the weaker assumption that the image of $\mathcal{A}$ in $k^*/k^{*e}$ is finite. We have chosen to make the stronger assumption as it simplifies the exposition and some technical aspects of the proofs.
\end{remark}

We finish with a simple example which solves the problem analogous to Theorem \ref{thm:existence} for field extensions of degree $n$ with maximal Galois group.

\begin{example} \label{ex:S_n}
Let $\alpha \in \QQ^*$ and $n\geq 3$. Then the polynomial 
      \begin{equation*}
          x^n+c x^{n-1}+t x+(-1)^{n}\alpha
        \end{equation*}
has Galois group $S_n$ over $\QQ(t)$ for all but finitely many $c\in \QQ $ (see \cite[Satz 1]{MR0276203}). Therefore, Hilbert's irreducibility theorem implies that for infinitely many specialisations $t \in \QQ$, the Galois group is $S_n$, and $\alpha$ is clearly a norm from such an extension, being the product of the roots of the defining polynomial.
\end{example}

\subsection{Methodology and structure of the paper} \label{sec:structure}
In \S\ref{sec:frob} we recall some of the theory of \emph{frobenian functions} from Serre's book \cite[\S 3.3]{Ser12}, in order to help analyse the Dirichlet series which arise in this paper.

In \S \ref{sec:main} we prove our main technical result, Theorem \ref{thm:main}. This is a general theorem for counting abelian extensions with local conditions imposed.  To prove this  we study the analytic properties of the Dirichlet series corresponding to our counting functions. We achieve this with the help of the harmonic analysis techniques developed in our earlier paper \cite{HNP}. In our case, however, the analysis is more difficult as the singularities of our Dirichlet series will be \emph{branch point} singularities, rather than poles, in general; this is reflected in the fact that $\varpi(k,G,\mathcal{A})$ in Theorem \ref{thm:global} can be a \emph{non-integral rational number}.  This section is the technical heart of the paper and is dedicated to the proof of Theorem \ref{thm:main}. 

Let us emphasise once more that we prove Theorem \ref{thm:existence} by first counting the extensions of interest and then showing that the leading constant obtained is positive. Our situation presents an interesting difficulty, however: the leading constant we obtain is not an Euler product but a \emph{sum} of Euler products and, in general, cancellation within these sums may occur for some choices of local conditions. For example, a famous theorem of Wang \cite{Wan50} says that there is no $\ZZ/8\ZZ$-extension of $\QQ$ which realises the unramified extension of $\QQ_2$ of degree $8$; in this case Wright observed in \cite[p.~48]{Wri89} that the Euler products appearing in the leading constant cancel out. We have to carefully analyse these sums of Euler products and explicitly show that no cancellation occurs in our case.

In \S \ref{sec:proofs}, we prove the major results stated in the introduction via suitable applications of Theorem \ref{thm:main} combined with Galois-cohomological techniques. At the end of \S \ref{sec:proofs} we also give a generalisation of Theorem \ref{thm:existence} which allows one to impose  local conditions on the abelian extension $K/k$ at finitely many places.

The appendix (by Yonatan Harpaz and Olivier Wittenberg) contains a purely geometric proof of Theorem \ref{thm:existence}. It uses descent and a version of the fibration method developed in \cite{HW18} to show that the Brauer--Manin obstruction controls the failure of weak approximation on a certain auxiliary variety. The existence of the required abelian extension is then shown using a version of Hilbert's irreducibility theorem due to Ekedahl \cite{ekedahl} (see also \cite[\textsection\textsection3.5--3.6]{serretopics}).

\subsection{Notation and conventions} \label{sec:notation}

We fix a number field $k$ throughout the paper and use the following notation:
\\ \vspace{-15pt} 
\begin{longtable}{ll}
$\Adele^*$ & the ideles of $k$\\
$\Adele_L^*$ & the ideles of a finite extension $L$ of $k$\\
$\OO_k$ &the ring of integers of $k$\\
$\Omega_k$ & the set of all places of $k$\\
$\OO_S$ &the $S$-integers of $k$\\
$v$ & a place of $k$\\
$k_v$ & the completion of $k$ at $v$\\
$\OO_v$ & the ring of integers of $k_v$. For $v\mid \infty$, by convention $\OO_v := k_v$\\
$\bbF_v$ & the residue field at a finite place $v$\\
$q_v$ & the cardinality of the residue field at a finite place $v$\\
$\zeta_k(s)$ & the Dedekind zeta function of $k$.\\
\end{longtable}

\noindent 
For locally compact abelian groups $A$ and $B$, we use the following notation:

\begin{longtable}{ll}
$\Hom(A,B)$ & the group of \emph{continuous} homomorphisms from $A$ to $B$, \\
& equipped with the
compact-open topology\\
$\dual{A}$ & the Pontryagin dual of $A$, $\dual{A} := \Hom(A, S^1)$\\
$\pair{\cdot}{\cdot}$ & the natural pairing $A\times \dual{A} \to S^1.$
\end{longtable}

\vspace{-0.3cm}

\noindent All finite groups are viewed as topological groups with the discrete topology.

For a place $v$ of $k$, a finite abelian group $G$, and $\chi \in
\Hom(k_v^*,G)$, we denote by $\Phi_v(\chi_v)$
the reciprocal of the $v$-adic norm of the conductor of $\Ker\chi_v$. For every $\chi\in\Hom(\Adele^*/k^*,G)$, we let $\Phi(\chi)$ be the reciprocal of the idelic norm of the conductor of the kernel of $\chi$; this equals the norm $\Phi(\varphi)$ of the  conductor of the sub-$G$-extension  $\varphi$ corresponding to $\chi$ via the global Artin map.

Let $K/k$ be an extension of number fields and $\alpha \in k^*$. We say that $\alpha$ is a \emph{(global) norm from $K$} if $\alpha \in \Norm_{K/k} K^*$. We say that $\alpha$ is a \emph{local norm at $v$ from $K$} if $\alpha \in \prod_{w\mid v} \Norm_{K_w/k_v} K_w^*\subset k_v^*$; if $K/k$ is Galois this is equivalent to the existence of \emph{some} place $w \mid v$ of $K$ such that $\alpha \in \Norm_{K_w/k_v} K_w^*$.

If $F$ is a field which contains $d$ distinct $d$th roots of unity and $\mathcal{A} \subset F^*$ is a finitely generated subgroup, then we denote by $F(\sqrt[d]{\mathcal{A}})$ the splitting field of the polynomials $x^d - \alpha$, where $\alpha$ runs over a set of generators of $\mathcal{A}$. 

For a subgroup $\mathcal{A} \subset k^*$ and a place $v$ of $k$, we denote by $\mathcal{A}_v$ the image of $\mathcal{A}$ in $k_v^*$.

\begin{ack} Our work on this project began at Michael Stoll's workshop \emph{Rational Points 2017} held at Franken-Akademie Schloss Schney. Substantial progress was made when the third author visited the other two at the University of Manchester, and also at the workshop \emph{Rational and Integral Points via Analytic and Geometric Methods} organised by Tim Browning, Ulrich Derenthal and Cec\'{i}lia Salgado at Hotel Hacienda Los Laureles, Oaxaca. We are very grateful to the organisers of both workshops, to the funding bodies, and to the staff at all three places for providing us with excellent working conditions. We thank the anonymous referee for a meticulous reading of an earlier draft of this paper and for several suggestions that helped improve it. The first-named author is supported by EPSRC-grant
EP/T01170X/2. The second-named author is supported by EPSRC grant EP/R021422/1 and UKRI Future Leaders Fellowship MR/V021362/1. The third-named author is supported by EPSRC grant EP/S004696/1 and UKRI Future Leaders Fellowship MR/T041609/1.
\end{ack}

\section{Frobenian functions} \label{sec:frob}

For the proofs of our main results, we will require some of the theory of frobenian functions, as can be found in Serre's book \cite[\S 3.3]{Ser12}.
Recall that a \emph{class function} on a group is a function which is constant on conjugacy classes.

\begin{definition} \label{def:frob}
	Let $k$ be a number field and 
	$\rho: \Val \to \CC$ a function on the set of places of $k$.
	Let $S$ be a finite set of places of $k$.
	We say that $\rho$ is $S$-\emph{frobenian} if there exist 
	\begin{enumerate}
		\item[(a)] a finite Galois extension $K/k$, with Galois group $\Gamma$, such that $S$ contains all places which ramify in $K/k$, and
		\item[(b)] a class function $\varphi: \Gamma \to \CC$,
	\end{enumerate}
	such that for all $v \not \in S$ we have
	$$\rho(v) = \varphi(\Frob_v),$$
	where $\Frob_v \in \Gamma$ denotes a Frobenius element of $v$.
	We say that $\rho$ is
        \emph{frobenian} if it is $S$-frobenian
	for some $S$.
	A subset of $\Val$ is called ($S$-)\emph{frobenian} if its indicator function is ($S$-)frobenian.	
\end{definition}

In Definition \ref{def:frob}, we adopt a common abuse of notation (see \cite[\S3.2.1]{Ser12}), and denote by $\Frob_v \in \Gamma$ the choice of some element of the Frobenius conjugacy class at $v$; note that $\varphi(\Frob_v)$ is well defined as $\varphi$ is a class function.

We define the \emph{mean} of $\rho$ to be 
$$m(\rho) = \frac{1}{|\Gamma|} \sum_{\gamma \in \Gamma}\varphi(\gamma)\in\CC.$$ 
	
\begin{example} \label{ex:frob_set}
	Let $f \in k[x]$ be a (not necessarily irreducible) polynomial. Then the
	set 
	$$\{ v \in \Omega_k : f(x) \text{ has a root in } k_v\}$$
	is frobenian. Indeed, take $K$ to be the splitting field of $f$.
	Then for a place $v$ which is unramified in $K$, the polynomial
	$f$ has a root in $k_v$ if and only if $\Frob_v$ acts with a fixed point on the roots of $f$ over $\bar{k}$;
	the set of such elements is a
	conjugacy invariant subset of the Galois group $\Gamma$.
\end{example}

We require the following result on the zeta function of a frobenian function. Throughout the paper, we write $q_v$ for the size of the residue field at a finite place $v$. Moreover, for any place $v$, let $\zeta_{k,v}(s)$ be the Euler factor of $\zeta_{k}(s)$ at $v$ if $v$ is non-archimedean, and $\zeta_{k,v}(s)=1$ otherwise.

\begin{proposition} \label{prop:frob_zeta}
	Let $S$ be a finite set of places of $k$ containing all archimedean places and
	let $\rho$ be an $S$-frobenian function. Assume that $|\rho(v)|<q_v$ holds for all $v\notin S$. Then the Euler product 
	\begin{equation} \label{def:Euler_S}
		F(s) = \prod_{v\notin S}\left(1 + \frac{\rho(v)}{q_v^s}\right)
              \end{equation}
              has the form
              \begin{equation}\label{eq:analyt_cont}
                F(s)=\zeta_k^{m(\rho)}(s)G(s),\quad\re s>1,                
              \end{equation}
for a function $G(s)$ that is holomorphic in a region 
              \begin{equation}\label{eq:SD}
                \re s > 1-\frac{c}{\log(|\im s|+3)},
              \end{equation}
              for some $c=c_\rho>0$, and satisfies in this region the bound
              \begin{equation}\label{eq:zero_free_bound}
                |G(s)|\ll_\rho (1+|\im s|)^{1/2}.
              \end{equation}
              Moreover, 
              \begin{equation}\label{eq:ep_at_one}
                \lim_{s\to 1}(s-1)^{m(\rho)}F(s)=(\mathrm{Res}_{s=1}\zeta_k(s))^{m(\rho)}\prod_{v\notin S}\frac{1+\rho(v)q_v^{-1}}{\zeta_{k,v}(1)^{m(\rho)}}\prod_{v\in S}\frac{1}{\zeta_{k,v}(1)^{m(\rho)}},
              \end{equation}
              and the limit in \eqref{eq:ep_at_one} is non-zero.
\end{proposition}
\begin{proof}
	First, note that the Euler factors $1+\rho(v)q_v^{-s}$ 
	are holomorphic on $\CC$ and non-zero for $\re s \geq 1$, 
	as $|\rho(v)| < q_v$ by assumption.
	Next, recall that the irreducible characters of a finite group $\Gamma$
	form a basis for the space of complex class functions of $\Gamma$ 
	\cite[Prop.~2.30]{FH91}.
	In particular, if $\varphi:\Gamma\to \CC$ is the class function associated to $\rho$, then we may write
	$$\varphi = \sum_{\chi} \lambda_\chi \chi$$
	where $\lambda_\chi \in \CC$ and the sum runs over the irreducible
	characters of $\Gamma$. For $\re s>1$, we find that
	\begin{align*}
          F(s)
	&= \prod_{v \notin S}
	\left(1 + \frac{ \sum_{\chi} \lambda_\chi \chi(\Frob_v)}{q_v^s} \right) = G_1(s)\prod_{\chi}L(\chi,s)^{\lambda_\chi},
	\end{align*}
	where $L(\chi,s)$ denotes the Artin $L$-function of $\chi$ and $G_1(s)$ is a holomorphic function with absolutely convergent Euler product on $\re s > 1/2$, which is non-zero on $\re s \geq 1$.
	
	For the trivial character $\chi=\one$, we have $L(\one,s) = \zeta_k(s)$. Since $\lambda_\one = m(\rho)$, we get the equality \eqref{eq:analyt_cont} with
        \begin{equation*}
G(s)=G_1(s)\prod_{\chi\neq\one}L(\chi,s)^{\lambda_\chi}.
\end{equation*}
By the Brauer induction theorem \cite[Thm.~VIII.7, p.~225]{CF}, we may decompose each remaining $L(\chi,s)$ as a product of $\ZZ$-powers of Hecke $L$-functions of non-trivial Hecke characters of subfields of $K$.
Hence, we assume from now on that each $L(\chi,s)$ is an entire Hecke $L$-function (for some possibly different number field). By \cite[Thm.~5.35]{IK04}, $L(\chi,s)$ respects a zero-free region of the form \eqref{eq:SD}, for some $c<1/4$ that may depend on $\chi$. Since there are only finitely many characters to consider, we can find a constant $c$ that works for all of them. Decreasing $c$ further, we obtain a bound
\begin{equation*}
  \log|L(\chi,s)|\ll \log\log(|\im s|+3), 
\end{equation*}
valid in the region \eqref{eq:SD} (cf. \cite[p.230]{Ser76}).
Using this bound and the fact that $|G_1(s)|\ll_\rho 1$ in $\re s\geq 3/4$ due to
the absolute convergence of its Euler product, it is simple to verify that $G(s)$ satisfies \eqref{eq:zero_free_bound}.

To verify \eqref{eq:ep_at_one}, we start with the following fact, which is well known at least in the classical case of Dirichlet $L$-functions: for non-trivial $\chi$, the Euler product of $L(\chi,s)$ converges for $s=1$ and takes the value $L(\chi,1)$. To see this, observe that $\log L(\chi,s)$ can be defined for $\re s>1$ as a Dirichlet series, use the prime number theorem for $L(\chi,s)$ (see \cite[Thm.~5.13]{IK04}) and partial summation to verify that this Dirichlet series converges for $s=1$, and apply Abel's theorem.

Since $G_1(s)$ has an absolutely convergent Euler product for $\re s>1/2$, this shows that the Euler product of $\zeta_k(s)^{-m(\rho)}F(s)=G(s)$ does indeed converge at $s=1$ and takes the value
\begin{equation*}
G(1)=\lim_{s\to 1}\zeta_k(s)^{-m(\rho)}F(s)=(\mathrm{Res}_{s=1}\zeta_k(s))^{-m(\rho)}\lim_{s\to 1}(s-1)^{m(\rho)}F(s).
\end{equation*}
Recalling our assumption that $|\rho(v)|<q_v$, it is clear that the right-hand side of \eqref{eq:ep_at_one} is non-zero.
\end{proof}

\begin{remark} \hfill
\begin{enumerate}
	\item Note that frobenian functions are bounded; thus the condition
	 $|\rho(v)| < q_v$ in Proposition \ref{prop:frob_zeta}
	is always satisfied for all but finitely many $v$.
	
	\item The conclusion
	\eqref{eq:ep_at_one} may fail if one includes the places $v\in S$ in the Euler product in Proposition~\ref{prop:frob_zeta}. To see this, take $k = \QQ, \rho(2) = -2$ and $\rho(p) = 0$
	for $p \neq 2$; this is frobenian with $K = \QQ$ and $S = \{2\}$.
	Then the  Euler factor
	$$1 - \frac{2}{2^s}$$
	has a zero at $s=1$, despite the fact that $m(\rho) = 0$. 
	
	\item The conclusion \eqref{eq:ep_at_one} can fail to hold for some
	innocuous looking Dirichlet series. Consider for example
	$F(s) = \zeta(2s -1)/\zeta(s)$. Then $\lim_{s\to 1} F(s) = 1/2$,
	but $\prod_p \lim_{s \to 1}(1-p^{-s})/(1-p^{-2s+1}) = 1$.
\end{enumerate}	
\end{remark}

\section{Counting with local conditions} \label{sec:main}

All of the main counting results in this paper are obtained from a more general counting
result, which we present in this section. To state this result we require some notation.

\subsection{Statement of the result} \label{sec:statement_main}
Let $G$ be a finite abelian group, let $F$ be a field and $\bar{F}$ a separable closure of $F$. We define a \emph{sub-$G$-extension} of $F$ to be a continuous homomorphism $\Gal(\bar{F}/F)\to G$. A sub-$G$-extension corresponds to a pair $(L/F,\psi)$, where $L/F$ is a Galois extension inside $\bar{F}$ and $\psi$ is an injective homomorphism $\Gal(L/F)\to G$. 

For each place $v$ of the number field $k$, we fix an algebraic closure $\bar{k}_v$ and compatible embeddings $k\hookrightarrow\bar{k}\hookrightarrow\bar{k}_v$ and $k\hookrightarrow k_v \hookrightarrow \bar{k}_v$. 

Hence, a sub-$G$-extension $\varphi$ of $k$ induces a sub-$G$-extension $\varphi_v$ of $k_v$ at every place $v$. For each place $v$ of $k$, let $\Lambda_v$ be a set of sub-$G$-extensions of $k_v$. For $\Lambda := (\Lambda_v)_{v \in \Omega_k}$ we are interested in the function
\begin{equation} \label{def:counting_conditions}
N(k,G,\Lambda,B) := \#\left\{\varphi\in\gextk:\, \Phi(\varphi)\leq B, \, \varphi_v \in \Lambda_v \forall v \right\},
\end{equation}
which counts those $G$-extensions of $k$ of bounded conductor which satisfy the local conditions imposed by $\Lambda$ at all places $v$. (Here $\Phi$ is as in \S\ref{sec:notation}.)

In general, it is difficult to say anything about the counting function given in \eqref{def:counting_conditions}, especially when there are infinitely many local conditions imposed. Even in the case when one imposes finitely many conditions, the set being counted may be empty, as explained in \S \ref{sec:structure}.
Our main technical result imposes arbitrary conditions at finitely many places, but at the remaining places we only impose those conditions which force every element of $\mathcal{A}$ to be a local norm.

\begin{theorem} \label{thm:main}
	Let $k$ be a number field, $G$ a non-trivial finite abelian group, and $\mathcal{A} \subset  k^*$ a finitely generated subgroup.
	Let $S$ be a finite set of places of $k$ and for $v \in S$ let 
	$\Lambda_v$ be a non-empty set of sub-$G$-extensions of $k_v$.
	For $v \notin S$
	we let $\Lambda_v$ be the set of sub-$G$-extensions of $k_v$ determined by
	those extensions of local fields $L/k_v$ for which every element of $\mathcal{A}$ is a local norm from $L/k_v$.
	Let $\Lambda := (\Lambda_v)_{v \in \Omega_k}$. Then there exist
	$ c_{k,G,\Lambda} \geq 0$ and $\delta=\delta(k,G,\mathcal{A}) >0$
	such that
	\begin{equation*}
		N(k,G,\Lambda,B) = c_{k,G,\Lambda}B(\log B)^{\varpi(k,G,\mathcal{A})-1}
		+  O(B(\log B)^{\varpi(k,G,\mathcal{A})-1-\delta}), \quad B\to\infty,
	\end{equation*}
	where 
	$\varpi(k,G,\mathcal{A})$ is as in Definition \ref{def:varpi_intro}.
	Moreover we have $c_{k,G,\Lambda} >0$ if there exists a 
	sub-$G$-extension of $k$
	which realises the given local conditions for all places $v$.
\end{theorem}

The leading constant $c_{k,G,\Lambda}$ in this theorem is given by a finite sum of Euler products (see Theorem \ref{thm:leading_constant} for an explicit expression). Our condition for positivity is only the existence of some \emph{sub}-$G$-extension of $k$ which realises the given local conditions; we do not require the existence of a genuine $G$-extension of $k$, so we do not need to assume that the set of $G$-extensions being counted is non-empty to deduce the positivity of the constant. This means that one need only look for an extension with possibly smaller Galois group to prove positivity of the constant; we use this trick to great effect when proving Theorem \ref{thm:existence}.

We illustrate how one applies Theorem~\ref{thm:main} in some simple cases. Firstly, one counts the total number of $G$-extensions of $k$ by applying Theorem~\ref{thm:main} with $\mathcal{A} = \{1\}$ and no local conditions, i.e.~taking $\Lambda_v$ to be the set of all sub-$G$-extensions of $k_v$ for all places $v$. These local conditions are realised by the sub-$G$-extension given by the trivial extension $k/k$. For a more interesting example,  consider the case $\mathcal{A} = \{1\}$ and the trivial local conditions $\Lambda_v = \{\one\}$ for $v \in S$, which are again  realised by the trivial extension $k/k$. 
This gives the following corollary.  (Note that we do not need to avoid the places above $2$.)

\begin{corollary}
	Let $S$ be a finite set of places. Then a positive proportion
	of $G$-extensions of $k$, ordered by conductor,
	are completely split at all places in $S$.
\end{corollary}

The rest of this section is dedicated to the proof of Theorem \ref{thm:main}.
All implied constants in the $O$ and $\ll$ notation are allowed to depend on $k,G, \mathcal{A}$ and $\Lambda$.

\subsection{The set of places $S$} \label{sec:S}
To prove Theorem \ref{thm:main}, we are free to increase the size of $S$ if we wish. Henceforth, we will assume that $S$ contains all archimedean places of $k$ and all places of $k$ lying above the primes
$p\leq |G|$,
that $\mathcal{A} \subset \OO_S^*$, and that $\OO_{S}$ has trivial class group.

The reader should note that many of the formulae which follow are only valid
for finite sets of places $S$ which satisfy these conditions.  For example, in
the case where $k=\QQ$, $G=\ZZ/8\ZZ$, $\mathcal{A}=\{1\}$, $S=\emptyset$, the
expression for the leading constant in Theorem \ref{thm:leading_constant} does
not hold. To compute $c_{k,G,\Lambda}$ in this instance, we may take
$S=\{\infty,2,3,5\}$ instead.

\subsection{Dirichlet series}
To prove Theorem \ref{thm:main} we study the associated Dirichlet
series
\begin{equation} \label{def:F_Lambda}
F_\Lambda(s)= \sum_{\varphi\in G\textrm{-ext}(k)}\frac{f_\Lambda(\varphi)}{\Phi(\varphi)^s},
\end{equation}
with $f_\Lambda$ the indicator function of those sub-$G$-extensions $\varphi\in\Hom(\Gal(\bar{k}/k),G)$ 
for which $\varphi_v\in\Lambda_v$ for all $v\in\Val$. Hence, $f_{\Lambda}$ is the product of the local indicator functions $f_{\Lambda_v}$ of $\Lambda_v$. As $|f_\Lambda(\varphi)|\leq 1$, this Dirichlet series defines a holomorphic function on $\re s>1$. (This follows from \cite[Lem.~2.10]{Woo10}, but also from the analysis later in this paper.)

\subsubsection{M\"{o}bius inversion} Recall that a $G$-extension of $k$ is a surjective continuous homomorphism $\varphi:\Gal(\bar{k}/k) \to G$. The condition that $\varphi$ be surjective is difficult to deal with, hence we perform a M\"{o}bius inversion to remove it. Let $\mu$ be the M\"obius function on isomorphism classes of finite abelian groups. That is, $\mu(G) = 0$ if $G$ has a cyclic subgroup of order $p^n$ with $p$ a prime and $n\geq 2$, $\mu(G_1\times G_2) = \mu(G_1)\mu(G_2)$ if $G_1$ and $G_2$ have coprime order, and $\mu((\ZZ/p\ZZ)^n) = (-1)^np^{n(n-1)/2}$ for a prime $p$ and $n \in \ZZ_{\geq 0}$. Let $f$ be a function on the subgroups of $G$. For subgroups $H\subset G$, we consider the function
\begin{equation*}
  g(H)=\sum_{J\subset H}f(J),
\end{equation*}
where the sum runs over all subgroups $J\subset H$. The M\"obius inversion formula for finite abelian groups \cite{Delsarte} states that
\begin{equation}\label{eq:mobius}
  f(G)=\sum_{H\subset G}\mu(G/H)g(H).
\end{equation}

\begin{lemma}\label{lem:mobius}
We have
\[F_\Lambda(s)= \sum_{H \subset G}\mu(G/H) 
\sum_{\varphi\in \Hom(\Gal(\bar{k}/k),H)}\frac{f_\Lambda(\varphi)}{\Phi(\varphi)^s}.\]
\end{lemma}

\begin{proof}
Sorting the sub-$H$-extensions $\varphi:\Gal(\bar k/k)\to H$ by their images, we get
  \begin{equation*}
\sum_{J\subset H}\sum_{\varphi\in J\textrm{-ext}(k)}\frac{f_\Lambda(\varphi)}{\Phi(\varphi)^s} = \sum_{\varphi\in\Hom(\Gal(\bar k/k),H)}\frac{f_\Lambda(\varphi)}{\Phi(\varphi)^s}.
\end{equation*}
Call the right-hand side $g(H)$ and apply M\"obius inversion \eqref{eq:mobius}.
\end{proof}
We now consider the contribution to $F_{\Lambda}(s)$ of each subgroup $H$ in turn. The contribution from $H=\{1\}$ is either $0$ or $1$. From now on we focus on the contributions of the non-trivial subgroups $H$.

\subsubsection{Class field theory}
Via global class field theory, we make the identification 
\begin{equation}\label{eq:cft}
  \Hom(\Gal(\bar{k}/k),H) = \Hom(\Adele^*/k^*,H).
\end{equation}
The canonical isomorphism \eqref{eq:cft} is induced by the global Artin map
$\Adele^*/k^* \to \Gal(k^{\textrm{ab}}/k)$. Using this isomorphism, we consider $f_{\Lambda}$ now as a function on $\Hom(\Adele^*/k^*,H)$. For every $\chi\in\Hom(\Adele^*/k^*,H)$, let $\Phi(\chi)$ be the reciprocal of the idelic norm of the conductor of the kernel of $\chi$, which is precisely the norm of the  conductor of the sub-$H$-extension corresponding to $\chi$. Together with Lemma \ref{lem:mobius}, this discussion shows the following:

\begin{lemma} \label{lem:CFT}
We have
\[F_\Lambda(s)= \sum_{H \subset G}\mu(G/H) 
\sum_{\chi \in \Hom(\Adele^*/k^*,H)}\frac{f_\Lambda(\chi)}{\Phi(\chi)^s}.\]
\end{lemma}
Hence, in our analysis of $F_\Lambda(s)$ we can now focus on the inner sums
 \[\sum_{\chi\in \Hom(\Adele^*/k^*, H)}\frac{f_{\Lambda}(\chi)}{\Phi(\chi)^s}.\]
Our counting problem fits very well within the class-field-theoretic framework. For each place $v\in\Val$, we use local class field theory (specifically, the local Artin map $k_v^* \to \Gal(k_v^{\textrm{ab}}/k_v)$) to make the identification
\begin{equation*}
  \Hom(\Gal(\bar k_v/k_v),H)=\Hom(k_v^*,H).
\end{equation*}
Thus, we consider $\Lambda_v$ as a subset of $\Hom(k_v^*,H)$. By the compatibility of local and global class field theory, we still have $f_{\Lambda}=\prod_vf_{\Lambda_v}$, with $f_{\Lambda_v}$ the indicator function of $\Lambda_v$.

\begin{lemma} \label{lem:kernel}
	Let $v \notin S$ and let $\chi_v\in\Hom(k_v^*, G)$. Then 
	$$f_{\Lambda_v}(\chi_v) = 1 \quad \Leftrightarrow \quad
	\mathcal{A}_v \subset \Ker \chi_v.$$
\end{lemma}
\begin{proof}
	Let $\varphi_v$ be the sub-$G$-extension of $k_v$ associated to
	$\chi_v$.
	By local class field theory  we have
	$$\Ker \chi_v = \Norm_{K_{\varphi_v}/k_v} K_{\varphi_v}^*,$$
	where $K_{\varphi_v}$ is the extension field of $k_v$ associated to
	$\varphi_v$. However, as $v \notin S$, by assumption in 
	Theorem \ref{thm:main} we have 
	$f_{\Lambda_v}(\chi_v) = 1$ if and only if every element of $\mathcal{A}$
	is a local norm	from $K_{\varphi_v}$; the result follows.
\end{proof}

\subsection{Harmonic analysis}

To deal with the sums  $$\sum_{\chi \in \Hom(\Adele^*/k^*,H)}\frac{f_\Lambda(\chi)}{\Phi(\chi)^s}$$
we shall use a version of the Poisson summation formula from harmonic analysis. The theory relevant to us was worked out in detail in \cite[\S 3]{HNP} when counting by \emph{discriminant}. The same theory transfers almost verbatim to show the validity of the Poisson summation formula for counting by conductor.

However, for the purposes of Theorem \ref{thm:main}, our case is special enough that we merely require a simplified version of the Poisson summation formula that can be proved using only character orthogonality for finite abelian groups. We may therefore forego some of the general theory from \cite[\S 3]{HNP}  and proceed in a more explicit manner. We first recall the set-up for the harmonic analysis.

\subsubsection{Fourier transforms} \label{sec:Fourier_def}

The group $\Hom(\Adele^*/k^*,H)$ is locally compact. Its Pontryagin dual is naturally identified with $\Adele^*/k^* \otimes H^\wedge$ (see \cite[\S3.1]{HNP}). We denote the associated pairing by $\langle \cdot, \cdot \rangle : \Hom(\Adele^*/k^*,H) \times (\Adele^*/k^* \otimes H^\wedge) \to S^1.$ Similarly, the Pontryagin dual of $\Hom(k_v^*,H)$ is naturally identified with $k_v^* \otimes \dual{H}$, and we also denote the relevant Pontryagin pairing by  $\langle \cdot, \cdot \rangle$. For each place $v$, we equip the finite group $\Hom(k^*_v,H)$ with the unique Haar measure $\mathrm{d}\chi_v$ such that 
$$\vol(\Hom(k_v^*/\OO_v^*,H)) = 1.$$
If $v$ is non-archimedean, this is $|H|^{-1}$ times the counting measure; for archimedean
$v$, recalling our convention that $\OO_v = k_v$, we obtain the counting measure.
The product of these measures yields a well-defined measure $\mathrm{d}\chi$ on $\Hom(\Adele^*,H)$. We say that an element of $\Hom(k^*_v,H)$ is \emph{unramified} if it lies in the subgroup $\Hom(k_v^*/\OO_v^*,H)$, i.e. if it is trivial on $\OO_v^*$, and that it is \emph{tamely ramified} if it is ramified and trivial on $1+\pi_v\OO_v$.

The function $f_\Lambda/\Phi^s$ is a product of local functions $f_{\Lambda_v}/\Phi_v^s$ on $\Hom(k_v^*,H)$, where $\Phi_v(\chi_v)$ is the reciprocal of the $v$-adic norm of the conductor of $\Ker\chi_v$. For $v\notin S$, these local functions take only the value $1$ on the unramified elements by our choice of $S$ and Lemma \ref{lem:kernel}, and thus $f_\Lambda/\Phi^s$ extends to a well-defined and continuous function on $\Hom(\Adele^*,H)$. We define its Fourier transform to be
$$\widehat{f}_{\Lambda,H}(x;s) = \int_{\chi \in \Hom(\Adele^*,H)}
\frac{f_\Lambda(\chi) \langle \chi, x \rangle}{\Phi(\chi)^s} \mathrm{d}\chi,$$
where $x = (x_v)_v \in  \Adele^* \otimes H^\wedge$.  Similarly, for $x_v \in k_v^{*} \otimes \dual{H}$ we have the local Fourier transform
$$\widehat{f}_{\Lambda_v,H}(x_v;s) = \int_{\chi_v \in \Hom(k_v^*,H)}
\frac{f_{\Lambda_v}(\chi_v) \langle \chi_v, x_v \rangle}{\Phi_v(\chi_v)^s} \mathrm{d}\chi_v.$$
For $\re s \gg 1$, the global Fourier transform exists and defines a holomorphic function in this domain, and there is an Euler product decomposition
\begin{equation} \label{eqn:Euler_product}
	\widehat{f}_{\Lambda,H}(x;s) = \prod_v \widehat{f}_{\Lambda_v,H}(x_v;s).
\end{equation}

\subsubsection{The local Fourier transforms}
Let $v\in\Val$ and $x_v\in k_v^*\otimes \dual{H}$.

\begin{lemma} \label{lem:holomorphic}
	The local Fourier transform 
	$\widehat{f}_{\Lambda_v,H}(x_v;s)$ is holomorphic on $\CC$ and satisfies $\widehat{f}_{\Lambda_v,H}(x_v;s)\ll_{k,H} 1$ on $\re s\geq 0$. 
	Moreover $\widehat{f}_{\Lambda_v,H}(1;s)>0$ for $s \in \RR$.
\end{lemma}
\begin{proof}
	We prove the result when $v$ is non-archimedean, the case of archimedean
	$v$ being analogous. By our choice of measures, we have
	\begin{align}\label{eq:localfourier}
		\widehat{f}_{\Lambda_v,H}(x_v;s) &= \frac{1}{|H|}
		\sum_{\chi_v \in \Hom(k_v^*,H)}
                                                 \frac{f_{\Lambda_v}(\chi_v) \langle \chi_v, x_v \rangle}{\Phi_v(\chi_v)^s}.
        \end{align}
This finite sum clearly defines a holomorphic function on $\CC$. If $\re s\geq 0$ then the sum is $\ll_{k,H} 1$, since every summand is bounded absolutely and the number of summands is $\ll_{k,H} 1$. For the last part, we have
    \begin{align*}
        \widehat{f}_{\Lambda_v,H}(1;s) &= \frac{1}{|H|}
		\sum_{\substack{\chi_v \in \Lambda_v}}
		\frac{1}{\Phi_v(\chi_v)^s}.
	\end{align*}
	For $v \in S$ the set $\Lambda_v$ is non-empty by assumption. For
	$v \notin S$ the set $\Lambda_v$ is again non-empty, as it 
	always contains the trivial homomorphism $k_v^* \to H$ by
	Lemma \ref{lem:kernel}. 
	For $s \in \RR$, we therefore obtain a
	finite non-empty sum of positive real numbers, which is positive.
\end{proof}

Now let $v$ be non-archimedean. Choosing a uniformiser of $k_v$ identifies $k_v^*/\cO_v^*$ with $\bbZ$ and gives a splitting of the exact sequence
\begin{equation} \label{seq:uniformiser}
	1\to \cO_v^*\to k_v^*\to k_v^*/\cO_v^*\to 1.
\end{equation}
This implies that the sequence 
\[1\to \Hom(k_v^*/\cO_v^*, H)\to \Hom(k_v^*, H)\to \Hom(\cO_v^*, H)\to 1\] 
is split exact. Thus
\begin{eqnarray}\label{eq:EulerFactor2}
\widehat{f}_{\Lambda_v,H}(x_v;s)=\frac{1}{|H|}\sum_{\psi_v\in \Hom(k_v^*/\cO_v^*, H)}\sum_{\chi_v\in \Hom(\cO_v^*, H)}\frac{f_{\Lambda_v}(\psi_v\chi_v)\langle \psi_v\chi_v, x_v \rangle}{\Phi_v(\chi_v)^s},
\end{eqnarray}
since $\psi_v$ is unramified and hence $\Phi(\psi_v\chi_v)=\Phi_v(\chi_v)$.

\begin{lemma}\label{lem:invariance}
Let $v \notin S$. Then $f_{\Lambda_v}$ is $\Hom(k_v^*/\cO_v^*, H)$-invariant and, in particular, $f_{\Lambda_v}(\psi_v) = 1$ for all $\psi_v \in \Hom(k_v^*/\cO_v^*, H)$.
\end{lemma}
 \begin{proof}	
	Let $\chi_v\in\Hom(k_v^*, H)$ and let $\psi_v\in \Hom(k_v^*/\cO_v^*, H)$. We use the criterion from Lemma \ref{lem:kernel}. We have $\mathcal{A}_v \subset \OO_v^*\subset\Ker \psi_v$. Therefore $\mathcal{A}_v \subset \Ker\psi_v\chi_v$ if and only if $\mathcal{A}_v \subset \Ker \chi_v$, whence $f_{\Lambda_v}(\psi_v\chi_v)=f_{\Lambda_v}(\chi_v)$, as required. With $\chi_v=\one$, this also shows the second assertion.
      \end{proof}
      
In the statement of the following lemma, note that the natural map $\OO_v^* \otimes \dual{H} \to k_v^* \otimes \dual{H}$ is injective, as the sequence \eqref{seq:uniformiser} is split exact. Therefore, we may naturally view $\OO_v^* \otimes \dual{H}$ as a subgroup of $k_v^* \otimes \dual{H}$.

\begin{lemma} \label{lem:O_v_sum}
	Let $v \notin S$. Then
	$$\widehat{f}_{\Lambda_v,H}(x_v;s) = 
	\begin{cases}
	\displaystyle{\sum_{\chi_v \in \Hom(\OO_v^*,H)}
	\frac{f_{\Lambda_v}(\chi_v) \langle \chi_v, x_v \rangle}{\Phi_v(\chi_v)^s}}, &  \text{if }x_v \in \OO_v^* \otimes \dual{H}, \\
	0, & \text{otherwise}.
	\end{cases}$$
\end{lemma}
\begin{proof}
	From \eqref{eq:EulerFactor2} and Lemma \ref{lem:invariance}
	we have
	\begin{align*}
          \widehat{f}_{\Lambda_v,H}(x_v;s)
	& = \frac{1}{|H|}\sum_{\chi_v\in \Hom(\cO_v^*, H)} \frac{f_{\Lambda_v}(\chi_v)\langle \chi_v, x_v \rangle}{\Phi_v(\chi_v)^s}  \sum_{\psi_v\in \Hom(k_v^*/\cO_v^*, H)} \langle \psi_v, x_v \rangle.
\end{align*}
	Now character orthogonality gives
	$$\sum_{\psi_v\in \Hom(k_v^*/\cO_v^*, H)} \langle \psi_v, x_v \rangle
	= \begin{cases}
		|\Hom(k_v^*/\cO_v^*, H)|, & \text{ if } x_v \in \OO_v^* \otimes \dual{H},\\
		0,& \text{otherwise}.
	\end{cases}$$
	Indeed, the subgroup $\OO_v^* \otimes \dual{H} \subset k_v^* \otimes \dual{H}$
	is naturally identified with the Pontryagin dual of $\Hom(k_v^*/\cO_v^*, H)$.
	The result now follows on noting that $k_v^*/\cO_v^* \cong \ZZ$
	and hence $|\Hom(k_v^*/\cO_v^*, H)| = |H|$.
\end{proof}

\subsubsection{Poisson summation}
We now prove the version of Poisson summation that we will require. 
In the statement, we view $\OO_S^*\otimes \dual{H}$ as a subgroup of $k^*\otimes \dual{H}$ as follows: we have the exact sequence
\begin{equation} \label{seq:P(O_S)}
	0 \to \OO_S^* \to k^* \to P(\OO_S) \to 0
\end{equation}
where $P(\OO_S)$ denotes the group of non-zero principal fractional ideals of $\OO_S$. Since $P(\OO_S)$ is a free abelian group, we have $\Tor(P(\OO_S),\dual{H}) = 0$. Therefore applying $(\cdot) \otimes \dual{H}$ to  \eqref{seq:P(O_S)} we find that the map $\OO_S^*\otimes \dual{H} \to k^*\otimes \dual{H}$ is injective, as required.

\begin{proposition} \label{prop:Poisson}
	For $\re s > 1$ the Fourier	transform $\widehat{f}_{\Lambda,H}(\cdot;s)$
	exists and defines a holomorphic function on this domain.
	Moreover, we have the Poisson formula
    \begin{equation}\label{eq:prop_poisson}
	 \sum_{\chi \in \Hom(\Adele^*/k^{*} , H)} \frac{f_\Lambda(\chi)}{\Phi(\chi)^{s}}
	=\frac{1}{|\OO_k^*\otimes\dual{H}|} 
	\sum_{x \in \OO_S^*\otimes \dual{H}}\widehat{f}_{\Lambda,H}(x;s), \quad \re s > 1.
   \end{equation}
\end{proposition}

Note that the group $\OO_S^*\otimes \dual{H}$ is finite by Dirichlet's $S$-unit theorem; in particular the right-hand sum is finite.
\begin{proof}
	Let $x \in \OO_S^*\otimes \dual{H}$. Let $x_v$ denote its image
	in $k_v^* \otimes \dual{H}$. Recall that we have normalised
	our Haar measures on $\Hom(k_v^*,H)$ to be $|H|^{-1}$ times the counting
	measure for non-archimedean $v$, and equal to the counting measure for archimedean
	$v$. We let $S_{\mathrm{f}}$ be the set of non-archimedean
	places in $S$. Now Lemma~\ref{lem:O_v_sum}  and \eqref{eqn:Euler_product} give
	\begin{align*}
	\widehat{f}_{\Lambda,H}(x;s) & =
	\frac{1}{|H|^{|S_{\mathrm{f}}|}}\prod_{v \in S} \sum_{\chi_v \in \Hom(k_v^*,H)}
	\hspace{-10pt}
\frac{f_{\Lambda_v}(\chi_v) \langle \chi_v, x_v \rangle}{\Phi_v(\chi_v)^s}
	\prod_{v \notin S} \sum_{\chi_v \in \Hom(\OO_v^*,H)}
	\hspace{-10pt}
\frac{f_{\Lambda_v}(\chi_v) \langle \chi_v, x_v \rangle}{\Phi_v(\chi_v)^s} \\
	& = \frac{1}{|H|^{|S_{\mathrm{f}}|}} \sum_{\chi \in \Hom(\Adele^*_S, H)}
	\frac{f_{\Lambda}(\chi) \langle \chi, x \rangle}{\Phi(\chi)^s} 
	\end{align*}
	where $\Adele^*_S = \prod_{v \in S}k_v^* \times \prod_{v \notin S} \OO_v^*$.
	We now change the order of summation in the right-hand sum
	of \eqref{eq:prop_poisson} to obtain
	\begin{align*}
		\sum_{x \in \OO_S^*\otimes \dual{H}}\widehat{f}_{\Lambda,H}(x;s)
		& = \frac{1}{|H|^{|S_{\mathrm{f}}|}} \sum_{\chi \in \Hom(\Adele^*_S, H)}
		\frac{f_{\Lambda}(\chi)}{\Phi(\chi)^s} 
		\sum_{x \in \OO_S^*\otimes \dual{H}}
		\langle \chi, x \rangle.
	\end{align*}
As $\Adele^*_S$ and $\Adele^*_S/\OO_S^*$ are locally compact groups and their subgroups of $n$th powers are closed, an application of \cite[Lem.~3.2]{HNP} gives canonical isomorphisms of abelian groups $ \Hom(\Adele^*_S, H)\cong \dual{(\Adele^*_S\otimes \dual{H})}$ and $ \Hom(\Adele^*_S/\OO_S^*, H)\cong \dual{(\Adele^*_S/\OO_S^*\otimes \dual{H})}$. Therefore, we can view an element $\chi\in \Hom(\Adele^*_S, H)$ as a character of $\Adele^*_S\otimes \dual{H}$. It is easily seen that $\chi$ induces the trivial character on $\OO_S^*\otimes \dual{H}$ if and only if $\chi\in  \Hom(\Adele^*_S/\OO_S^*, H)$. 
	Thus, we may apply character orthogonality to find
	that
	$$\sum_{x \in \OO_S^*\otimes \dual{H}} \langle \chi, x \rangle
	= \begin{cases}
		|\OO_S^*\otimes \dual{H}|, & \text{ if }\chi \in  \Hom(\Adele^*_S/\OO_S^*, H),\\
		0,& \text{otherwise}.
	\end{cases}$$
	We therefore obtain
	$$\sum_{x \in \OO_S^*\otimes \dual{H}}\widehat{f}_{\Lambda,H}(x;s)
	= \frac{|\OO_S^*\otimes \dual{H}|}{|H|^{|S_{\mathrm{f}}|}}
	\sum_{\chi \in \Hom(\Adele^*_S/\OO_S^*, H)}
	\frac{f_{\Lambda}(\chi)}{\Phi(\chi)^s}.$$
	Dirichlet's $S$-unit theorem gives a (non-canonical) isomorphism $\OO_S^* \cong \OO_k^* \times \ZZ^{S_{\mathrm{f}}}$, whereby
	$$\frac{|\OO_S^*\otimes \dual{H}|}{|H|^{|S_{\mathrm{f}}|}} = |\OO_k^* \otimes H|.$$	
	Moreover, as $\OO_S$ has trivial class group, 
	the natural map $\Adele_S^*/\OO_S^*
	\to \Adele^*/k^*$ is an isomorphism \cite[Lem.~2.8]{Woo10}.
	The result now easily follows.	
\end{proof}

\subsection{Analytic continuation of the Fourier transforms} \label{sec:analytic_continuation}
We now use the Poisson formula to study the analytic behaviour of 
the Dirichlet series under consideration. To do so, we shall
calculate explicitly the local Fourier transforms for
$v \notin S$. Fix some subgroup $H$ of $G$. By a slight abuse of notation, for $x_v \in k_v^* \otimes \dual{H}$ we write $x_v\in \mathcal{A}_v \otimes\dual{H}$ to express that $x_v$ is in the image of the (not necessarily injective) map $\mathcal{A}_v \otimes \dual{H}\to k_v^*\otimes\dual{H}$. 

\begin{lemma} \label{lem:local_good}
	Let $v \notin S$ and let $x_v \in \OO_v^* \otimes \dual{H}$. 
	Then
	\begin{equation}\label{eq:goodlocalfactors}
	\widehat{f}_{\Lambda_v,H}(x_v;s)=
	\begin{cases}
	1+(|\Hom(\FF_v^*/(\mathcal{A} \bmod v), H)|-1)q_v^{-s}, & 
	\textrm{ if } x_v \in \mathcal{A}_v \otimes \dual{H},\\
	1-q_v^{-s}, &  \textrm{ if } x_v\notin \mathcal{A}_v \otimes \dual{H}.
	\end{cases} \nonumber
      \end{equation}
\end{lemma}

\begin{proof}
	An element $\chi_v\in \Hom(\cO_v^*, H)$ is unramified if and only if it is trivial. Furthermore, since $v\notin S$ and our assumptions on $S$ in \S \ref{sec:S}, the ramification is tame and hence for non-trivial characters $\chi_v\in \Hom(\cO_v^*, H)$, we have $\Phi_v(\chi_v)= q_v$. Therefore, by Lemmas \ref{lem:kernel} and \ref{lem:O_v_sum} we have
	\begin{align}	
		\widehat{f}_{\Lambda_v,H}(x_v;s) & = 
		\sum_{\chi_v \in \Hom(\OO_v^*,H)}
		\frac{f_{\Lambda_v}(\chi_v) \langle \chi_v, x_v \rangle}{\Phi_v(\chi_v)^s} \nonumber\\
		& = 1+ \sum_{\substack{\chi_v\in \Hom(\cO_v^*, H)\\ \chi_v\neq \mathbbm{1}}}\frac{f_{\Lambda_v}(\chi_v)\langle \chi_v, x_v \rangle}{q_v^s}\nonumber \\
		& =1+ q_v^{-s}\sum_{\substack{\chi_v\in \Hom(\cO_v^*, H)\\ \chi_v\neq \mathbbm{1} \\ \mathcal{A}_v \subset \Ker \chi_v}}\langle \chi_v, x_v \rangle\nonumber\\
		&=1-q_v^{-s}+ q_v^{-s}\sum_{\substack{\chi_v\in \Hom(\cO_v^*, H) \\  \mathcal{A}_v \subset \Ker \chi_v}}\langle \chi_v, x_v \rangle. \label{eqn:ker_a}
	\end{align}
	We claim that  the natural map
	\begin{equation} \label{eqn:Hom_iso}
		\Hom(\FF_v^*/(\mathcal{A} \bmod v), H) \to \{\chi_v \in \Hom(\cO_v^*, H) : \mathcal{A}_v \subset \Ker \chi_v\}
	\end{equation}
	is an isomorphism.
	To see this, recall that Hensel's lemma
	yields a split short exact sequence
	$$1 \to 1+\mathfrak{p}_v \to \OO_v^* \to \FF_v^* \to 1,$$
	where $\fp_v$
	denotes the maximal ideal of $\OO_v$.
	Applying $\Hom(\cdot, H)$, we obtain
	$$1 \to \Hom(\FF_v^*,H) \to \Hom(\OO_v^*,H) \to \Hom(1+\mathfrak{p}_v,H) \to 1.$$
	The kernel of a continuous homomorphism $1+\mathfrak{p}_v\to H$ contains $1+\mathfrak{p}_v^n$ for some $n\in\mathbb{N}$, and the successive quotients in the filtration $1+\mathfrak{p}_v\supset 1+\mathfrak{p}_v^2\supset \dots \supset 1+\mathfrak{p}_v^n$ each have order $|\cO_v/\mathfrak{p}_v|=q_v$ (see \cite[Prop.~IV.2.6]{Ser79}).
	Consequently, the quotient $(1+\mathfrak{p}_v)/(1+\mathfrak{p}_v^n)$ has order a power of $q_v$. Now recall that we assumed in \S\ref{sec:S} that $\gcd(q_v,H) =1$. Therefore, any continuous homomorphism $1+\mathfrak{p}_v\to H$ is trivial. It follows that
	$\Hom(\FF_v^*,H) = \Hom(\OO_v^*,H)$.

	Moreover, $\mathcal{A} \bmod v$ lies in the kernel of a
	homomorphism $\FF_v^* \to H$ if and only if $\mathcal{A}_v$ lies in the kernel of the induced
	homomorphism $\OO_v^* \to H$; whence \eqref{eqn:Hom_iso} is an isomorphism
	as claimed.
	
	Orthogonality of characters now gives
	$$\sum_{\substack{\chi_v\in \Hom(\cO_v^*, H) \\ \mathcal{A} \subset \Ker \chi_v}}
	\langle \chi_v, x_v \rangle =
	\begin{cases}
		|\Hom(\FF_v^*/(\mathcal{A} \bmod v), H)|,& \text{if }
		x_v\in \mathcal{A}_v \otimes \dual{H}, \\
		0,&  \text{otherwise}.
              \end{cases}$$
   Inputting this into \eqref{eqn:ker_a}, the result follows.
\end{proof}
To study the analytic behaviour of the global Fourier transforms $\widehat{f}_{\Lambda,H}(x;s)$, we use the theory of frobenian functions from \S\ref{sec:frob}.

\begin{lemma}\label{lem:d_a_H}
  Let $e$ be the exponent of $H$. Consider a function $d_{\mathcal{A},H}:\Val \to \CC$ which for all $v\notin S$  satisfies
	$$
d_{\mathcal{A},H}(v) = \max\{ d\in\ZZ: d \textrm{ divides } \gcd(e,q_v-1) \textrm{ and } \mathcal{A} \bmod v \subset \FF_v^{*d}\}.
	$$
  Then
  \begin{enumerate}
  \item $d_{\mathcal{A},H}$ is $S$-frobenian, and
  \item for $v\notin S$, we have $|\Hom(\FF_v^*/(\mathcal{A} \bmod v),H)| = |H[d_{\mathcal{A},H}(v)]|$.
  \end{enumerate}
\end{lemma}

\begin{proof}
  \emph{(1)} For every $d\mid e$, consider the number field $k_d=k(\zeta_d,\sqrt[d]{\mathcal{A}})$. The subset
  \begin{equation*}
    \Sigma_d = \Gal(k_e/k_d)\smallsetminus\bigcup_{\substack{d'\mid \frac{e}{d}\\d'\neq 1}}\Gal(k_e/k_{dd'}) \subset \Gal(k_e/k)
  \end{equation*}
  is a union of conjugacy classes, since each $\Gal(k_e/k_{d})$ is normal in $\Gal(k_e/k)$. The sets $\Sigma_d$ for $d\mid e$ form a partition of $\Gal(k_e/k)$. Let $\varphi : \Gal(k_e/k)\to\CC$ be the class function that takes the constant value $d$ on $\Sigma_d$, for all $d\mid e$. We claim that $d_{\mathcal{A},H}(v)=\varphi(\Frob_v)$ for all $v\notin S$, so in particular it is $S$-frobenian.

  Note that $\Frob_v\in \Sigma_d$ if and only if $d$ is the largest divisor of $e$ such that $v$ splits completely in $k_d/k$. Equivalently, $d$ is the largest divisor of $e$ such that $d\mid q_v-1$ and $x^d-\alpha$ has a root in $k_v$ for all $\alpha \in \mathcal{A}$. By Hensel's lemma, this is equivalent to $d=d_{\mathcal{A},H}(v)$, and thus $\varphi(\Frob_v)=d_{\mathcal{A},H}(v)$, as desired.
  
  \emph{(2)} Let $m$ be the largest divisor of $q_v-1$ such that $\mathcal{A}\bmod v \subset \FF_v^{*m}$. Then $\mathcal{A}\bmod v=\FF_v^{*m}$, and thus $\FF_v^*/(\mathcal{A} \bmod v)\cong\ZZ/m\ZZ$. Hence
  \begin{align*}
|\Hom(\FF_v^*/(\mathcal{A} \bmod v),H)|&=|\Hom(\ZZ/m\ZZ,H)|=|\Hom(\ZZ ,H[m])|\\ &=|H[m]|=|H[\gcd(m,e)]|=|H[d_{\mathcal{A},H}(v)]|.    \qedhere
  \end{align*}
\end{proof}

\begin{lemma}\label{lem:frob_set}
	Let $x \in \OO_S^* \otimes \dual{H}$. Then the set 
	$$
	\{v \in \Val : x_v \in \mathcal{A}_v \otimes \dual{H}\}
	$$
	is $S$-frobenian. 	In the special case $\dual{H} = \ZZ/e\ZZ$, on identifying
	$k^* \otimes \dual{H}$ with $k^*/k^{*e}$, this set equals
	\begin{equation} \label{eqn:eth_power}
	\left\{ v \in \Val :  \text{the polynomial }
	\prod_{\alpha \in \mathcal{A}/\mathcal{A}^e}
	(t^e - x\alpha)
	\text{ has a root in }k_v\right\}.
	\end{equation}
\end{lemma}
Note that the group $\mathcal{A}/\mathcal{A}^e$ is finite. Moreover, our slight abuse of notation is harmless, as whether or not 
the polynomial appearing in \eqref{eqn:eth_power} has a root is independent of the choice
of representative of each element of $\mathcal{A}/\mathcal{A}^e$.
\begin{proof}
	To prove this we choose a presentation of $\dual{H}$.
	We  then work coordinate-wise on $\dual{H}$, using the fact that the intersection
	of finitely many frobenian sets is frobenian.
	Thus, we reduce to the case $\dual{H} = \ZZ/e\ZZ$.
	Here we have $\OO_S^* \otimes \dual{H} = \OO_S^*/\OO_S^{*e}$.
	For $x \in \OO_S^*$, we have to show that the set
	$$ \{ v \in \Val : x_v \in \mathcal{A}_vk_v^{*e} \}$$
	is $S$-frobenian. However, we have $x_v \in \mathcal{A}_vk_v^{*e}$
	if and only if $x_v\alpha_v \in k_v^{*e}$ for some $\alpha_v \in \mathcal{A}_v$ (depending on $v$).
	We find that the set in question is
	the set of places $v$ such that the equation
	$$\prod_{\alpha \in \mathcal{A}/\mathcal{A}^e}
	(t^e - x\alpha) = 0$$
	has a solution in $k_v$; this set is  frobenian (see Example \ref{ex:frob_set}). As $x$ is an $S$-unit, it is easily seen that this is $S$-frobenian for our choice of $S$ in \S\ref{sec:S}.
\end{proof}

\begin{corollary} \label{cor:is_frob}
	Let $x \in \OO_S^* \otimes \dual{H}$. Then the function 
	$$
	v \mapsto 
	\begin{cases}
	|\Hom(\FF_v^*/(\mathcal{A} \bmod v), H)|-1, & 
	\textrm{ if } x_v \in 	\mathcal{A}_v \otimes \dual{H}, \\
	-1, &  \textrm{ if } x_v\notin \mathcal{A}_v \otimes \dual{H},
	\end{cases}
	$$
	is $S$-frobenian. 
\end{corollary}
\begin{proof}
	The product or sum of two $S$-frobenian functions is clearly
	$S$-frobenian (in Definition \ref{def:frob} one
	takes the compositum of the relevant field extensions).
	Moreover, the complement of a $S$-frobenian set is $S$-frobenian.
	The result therefore follows from Lemmas \ref{lem:d_a_H} and
	\ref{lem:frob_set}.
\end{proof}

\begin{definition}\label{def:is_frob_mean}
We denote by $\varpi(k,H,\mathcal{A},x)$ the mean of the $S$-frobenian function described in Corollary~\ref{cor:is_frob}.
\end{definition}

We now compare $\varpi(k,H,\mathcal{A},x)$ with $\varpi(k,H,\mathcal{A})$, as defined in Definition~\ref{def:varpi_intro}.

\begin{lemma} \label{lem:varpi_def}
	We have $\varpi(k,H,\mathcal{A},x)\leq\varpi(k,H,\mathcal{A})$ for all $x\in\OO_S^*\otimes\dual{H}$. Moreover,	$\varpi(k,H,\mathcal{A},1) = \varpi(k,H,\mathcal{A})$. 
\end{lemma}
\begin{proof}
	As clearly $\varpi(k,H,\mathcal{A},x)\leq\varpi(k,H,\mathcal{A},1)$,
 the first assertion follows immediately from the second. So let us prove the second assertion.
  By Corollary~\ref{cor:is_frob} and Lemma \ref{lem:d_a_H}, we see that $\varpi(k,H,\mathcal{A},1)$ is the mean of an $S$-frobenian function $\rho$ with  $\rho(v)=|\Hom(\FF_v^*/(\mathcal{A} \bmod v), H)|-1 = |H[d_{\mathcal{A},H}(v)]|-1$ for all $v\notin S$. With the notation of the proof of Lemma \ref{lem:d_a_H}, the corresponding class function on $\Gal(k_e/k)$ is given by $\sigma\mapsto |H[\varphi(\sigma)]|-1$. Hence, 
  \begin{equation*}
    \varpi(k,H,\mathcal{A},1) = \frac{1}{[k_e:k]}\sum_{\sigma\in\Gal(k_e/k)}(|H[\varphi(\sigma)]|-1) = \frac{1}{[k_e:k]}\sum_{d\mid e}(|H[d]|-1)|\Sigma_d|.
  \end{equation*}
  By inclusion-exclusion, we get $|\Sigma_d|=\sum_{c\mid \frac{e}{d}}\mu(c)\,|\Gal(k_e/k_{cd})|$,
  and thus
  \begin{align*}
    \varpi(k,H,\mathcal{A},1) &= \sum_{d\mid e}(|H[d]|-1)\sum_{c\mid \frac{e}{d}}\frac{\mu(c)}{[k_{cd}:k]}=\sum_{f\mid e}\frac{1}{[k_f:k]}\sum_{d\mid f}(|H[d]|-1)\mu(f/d)\\
    &=-1+\sum_{f\mid e}\frac{1}{[k_f:k]}\sum_{d\mid f}|H[d]|\mu(f/d)\\
    &=-1+\sum_{f\mid e}\frac{\#\{g\in H : |g|=f\}}{[k_f:k]} = \varpi(k,H,\mathcal{A}).
    \qedhere
  \end{align*}
\end{proof}

Recall that $\zeta_{k,v}(s)$ is the Euler factor of $\zeta_k(s)$ at a non-archimedean place $v$. If $v$ is archimedean, then we let $\zeta_{k,v}(s)=1$. 

\begin{proposition} \label{prop:Fourier_properties}
	Let $x \in \OO_S^* \otimes \dual{H}$. Then
	the Fourier transform satisfies 
	$$\widehat{f}_{\Lambda,H}(x;s) = \zeta_k(s)^{\varpi(k,H,\mathcal{A},x)}G(x;s),\quad\re s>1,$$ where $G(x;s)$ is holomorphic in the region \eqref{eq:SD}, for some $c>0$, and satisfies \eqref{eq:zero_free_bound}. Moreover, we have
        \begin{equation*}
          \lim_{s\to 1}(s-1)^{\varpi(k,H,\mathcal{A},x)}\widehat{f}_{\Lambda,H}(x;s)=(\mathrm{Res}_{s=1}\zeta_k(s))^{\varpi(k,H,\mathcal{A},x)}\prod_{v\in\Omega_k}\frac{\widehat{f}_{\Lambda_v,H}(x_v;1)}{\zeta_{k,v}(1)^{\varpi(k,H,\mathcal{A},x)}}.
        \end{equation*}
In the case $x=1$, this limit is non-zero.
\end{proposition}
\begin{proof}
	We consider the Euler product expansion of $\widehat{f}_{\Lambda,H}(x;s)$
	from \eqref{eqn:Euler_product}, where the Euler factors at $v\notin S$ were determined in Lemma \ref{lem:local_good}. By Corollary \ref{cor:is_frob} and our assumptions on $S$, we may apply Proposition \ref{prop:frob_zeta} to obtain
        \begin{equation*}
          F(s):=\prod_{v\notin S}\widehat{f}_{\Lambda_v,H}(x_v;s)=\zeta_k(s)^{\varpi(k,H,\mathcal{A},x)}\mathcal{H}(x;s),
        \end{equation*}
with a function $\mathcal{H}(x;s)$ that is holomorphic in a region \eqref{eq:SD} and satisfies the bound \eqref{eq:zero_free_bound}.
	By Lemma~\ref{lem:holomorphic}, we may multiply $\mathcal{H}(x;s)$ by the Euler factors $\widehat{f}_{\Lambda_v,H}(x_v;s)$ for $v\in S$ while still preserving these properties (possibly for a smaller $c>0$ in~\eqref{eq:SD}). Finally, the explicit form of the limit follows from \eqref{eq:ep_at_one} which, together with Lemma \ref{lem:holomorphic}, also shows that the limit is non-zero if $x=1$.
\end{proof}

\subsection{The asymptotic formula in Theorem \ref{thm:main}}
We now bring all our tools together to prove the first part of 
Theorem \ref{thm:main}.
Recall from Lemma \ref{lem:CFT} that we performed
a M\"{o}bius inversion to obtain a sum over the subgroups $H$ of $G$.
Moreover, in Proposition \ref{prop:Poisson} we used Poisson summation to understand
the inner sums from Lemma \ref{lem:CFT}.
In summary,
\begin{equation}\label{eq:dirichlet_series_sum}
	F_\Lambda(s) = \sum_{H \subset G} \frac{\mu(G/H)}{|\OO_k^* \otimes \dual{H}|} \sum_{x \in \OO_S^* \otimes \dual{H}} \widehat{f}_{\Lambda,H}(x;s), \quad \re s > 1,
      \end{equation}
where $F_{\Lambda}$ is the Dirichlet series from \eqref{def:F_Lambda}.
Furthermore, we described the analytic properties of the Fourier transforms $\widehat{f}_{\Lambda,H}(x;s)$ in  Proposition \ref{prop:Fourier_properties}.

By Lemma \ref{lem:local_good}, we can expand each of the Euler products $\widehat{f}_{\Lambda,H}(x;s)$ as a Dirichlet series
\begin{equation}\label{eq:def_dirichlet_coefficients}
  \widehat{f}_{\Lambda,H}(x;s) = \sum_{n\in\ZZ_{\geq 1}}\frac{a_n(H,x)}{n^s},
\end{equation}
with coefficients $a_n(H,x)\in\CC$.

\begin{lemma}\label{lem:selberg_delange}
  Let $H\subset G$ be a subgroup, let  $x\in\OO_S^{*}\otimes\dual{H}$, and let $a_n(H,x)$ be the Dirichlet coefficient from \eqref{eq:def_dirichlet_coefficients}. Then
  \begin{equation*}
    \sum_{n\leq B}a_n(H,x) = c_{H,x}B(\log B)^{\varpi(k,H,\mathcal{A},x)-1}
		+ O(B(\log B)^{\varpi(k,H,\mathcal{A},x)-2}),
              \end{equation*}
              where
              \begin{equation*}
              c_{H,x}=\frac{1}{\Gamma(\varpi(k,H,\mathcal{A},x))}\lim_{s\to 1}(s-1)^{\varpi(k,H,\mathcal{A},x)} \widehat{f}_{\Lambda,H}(x;s).
            \end{equation*}
\end{lemma}

\begin{proof}
  Let start by recalling that for every $\delta>0$ there is a value of $c=c(\delta)>0$
  such that $|\zeta_k(s)/\zeta(s)|\ll_\delta (|\Im s|+3)^\delta$ for all $s$ in the
  region \eqref{eq:SD}.  Indeed, in the case $\re(s)\geq 2$ a stronger bound
  follows directly from the fact that the Euler product of
  $\zeta_k(s)/\zeta(s)$ converges absolutely. In the compact region defined by
  $|\im(s)|\leq 2$ and $1-c/\log(|\Im s|+3)\leq \re s\leq 2$ for some small enough $c$,
  the function $\zeta_k(s)/\zeta(s)$ is holomorphic and thus bounded.  It
  remains to consider the case $|\im s|\geq 2$, where we stay away from the
  poles at $s=1$. It is well known that
  $\zeta(s)\neq 0$ and $|1/\zeta(s)|\ll \log|\im s|$ for small enough $c$
  (e.g. \cite[(3.11.8)]{Tit86}). Sufficient upper bounds for $|\zeta_k(s)|$
  follow from standard convexity bounds (e.g.~\cite[Theorem
  5.30]{IK04}).

  Write $\varpi = \varpi(k,H,\mathcal{A},x)$. Let $G(x;s)$ be as in Proposition
  \ref{prop:Fourier_properties}, and let the constant $c$ be small enough to
  ensure that $\zeta(s)\neq 0$, $\zeta_k(s)\neq 0$, and
  $|\zeta_k(s)/\zeta(s)|^{\varpi}\ll (|\im s|+3)^{1/4}$ for all $s$ in the
  region \eqref{eq:SD}. Then the function $h(s):=\zeta_k(s)^\varpi/\zeta(s)^{\varpi}$, defined on $\re s>1$ via the binomial series applied to the Euler factors, has an analytic continuation to the region \eqref{eq:SD}. Hence, the function $H(x;s) = h(s)G(x;s)$ is holomorphic and satisfies $H(x;s)\ll (|\im s|+3)^{3/4}$ in the region \eqref{eq:SD}.  Since $\widehat{f}_{\Lambda,H}(x;s)=\zeta(s)^{\varpi}H(x;s)$, we may apply the Selberg--Delange method in the form of \cite[Thm.~II.5.2]{Ten15} (with $N=0$) to obtain the required asymptotic. (For the sequence $(b_n)_n$ required in \cite[Thm.~II.5.2]{Ten15}, we take the coefficients $a_n(H,1)$. One can observe directly from the definition of the Euler factors $\widehat{f}_{\Lambda_v,H}(x_v;s)$ that these coefficients satisfy $a_n(H,1)\geq |a_n(H,x)|$.)
\end{proof}

Let us note that the leading term will come from $H=G$.

\begin{lemma} \label{lem:H<G}
	Let $H \subset G$ be a proper subgroup. Then $\varpi(k,H,\mathcal{A}) < \varpi(k,G,\mathcal{A})$.
\end{lemma}
\begin{proof}
	Follows immediately from Definition \ref{def:varpi_intro}.
\end{proof}

We are now finally in the position to prove the required asymptotic formula.

\begin{proposition}\label{prop:asymptotic}
Write $\varpi=\varpi(k,G,\mathcal{A})$.
  There exists $\delta=\delta(k,G,\mathcal{A}) > 0$
  such that 
	\begin{equation*}
		N(k,G,\Lambda,B) = c_{k,G,\Lambda}B(\log B)^{\varpi-1}
		+ O(B(\log B)^{\varpi-1-\delta}),
              \end{equation*}
  where
  \begin{equation*}
c_{k,G,\Lambda}=\frac{1}{\Gamma(\varpi)|\OO_k^*\otimes G|}\sum_{\substack{x\in\OO_S^*\otimes\dual{G}\\ \varpi(k,G,\mathcal{A},x)=\varpi}}\lim_{s\to 1}(s-1)^{\varpi}\widehat{f}_{\Lambda,G}(x;s).
\end{equation*}
\end{proposition}
\begin{proof}
  By \eqref{eq:dirichlet_series_sum} and \eqref{eq:def_dirichlet_coefficients}, the Dirichlet coefficients $f_n$ of $F_{\Lambda}(s)$ satisfy
  \begin{equation*}
    f_n = \sum_{H \subset G} \frac{\mu(G/H)}{|\OO_k^* \otimes \dual{H}|} \sum_{x \in \OO_S^* \otimes \dual{H}} a_n(H,x).
  \end{equation*}
  Since $N(k,G,\Lambda,B)=\sum_{n\leq B}f_n$, the proposition now follows from Lemmas \ref{lem:varpi_def}, \ref{lem:selberg_delange}, and \ref{lem:H<G}. 
\end{proof}
This proves the asymptotic formula in Theorem \ref{thm:main}. Next, we study the leading constant.

\subsection{Formula for the leading constant}

To calculate the leading constant, we first need to understand exactly which elements of $\OO_S^* \otimes \dual{H}$ give rise to the leading singularity in the Poisson sum (Proposition \ref{prop:Poisson}).

\begin{lemma} \label{lem:finitely_many}
	Let
	$$\mathcal{X}(k,G,\mathcal{A}) = \{ x \in k^* \otimes \dual{G} :
	x_v \in \mathcal{A}_v \otimes \dual{G} \text{ for all but finitely many } v\}.$$
	Then $\mathcal{X}(k,G,\mathcal{A})$ is finite
	and 
	$$\mathcal{X}(k,G,\mathcal{A}) = \{ x \in \OO_S^* \otimes \dual{G} :
	x_v \in \mathcal{A}_v \otimes \dual{G} \text{ for all } v \notin S\}.$$	
\end{lemma}
\begin{proof}
It is enough to prove the result for $\dual{G}$ a cyclic group of prime power order. Henceforth, let $\dual{G}=\ZZ/q\ZZ$, where $q=p^r$ is a prime power. We view $\mathcal{X}(k,G,\mathcal{A})$ as a subgroup of $k^*/k^{*q}$. 

First, we claim that the image of $\mathcal{X}(k,G,\mathcal{A})$ in $ k(\mu_{q})^*/k(\mu_{q})^{*q}$ is equal to the image of $\mathcal{A}$. One containment is clear, as  $\mathcal{A}\otimes \dual{G} \subset \mathcal{X}(k,G,\mathcal{A})$. For the other, let $K=k(\mu_{q},\sqrt[q]{\mathcal{A}})$, so $K=k_q$ in the notation of Definition~\ref{def:varpi_intro}. Let $K_v$ be the
	completion of $k$ at a choice of place of $K$ above $v$. The image of $\mathcal{X}(k,G,\mathcal{A})$ in $ k(\mu_{q})^*/k(\mu_{q})^{*q}$ is contained in the following set:
\[ \left\{x\in k(\mu_{q})^*/k(\mu_{q})^{*q} : x_v\in K_v^{*q}\text{ for all but finitely many } v\right\}.\]
As  $\mu_{q} \subset K$, an application of the Chebotarev density theorem shows that this set equals $(k(\mu_{q})^*\cap K^{*q})/k(\mu_{q})^{*q}$
(this also follows from Lemma~\ref{lem:Sha}). On the other hand, Kummer theory shows that $(k(\mu_{q})^*\cap K^{*q})/k(\mu_{q})^{*q}$ is equal to the image of $\mathcal{A}$ in $k(\mu_{q})^*/k(\mu_{q})^{*q}$, and  the claim is proved. In particular, the image $\mathcal{X}(k,G,\mathcal{A})$ in $ k(\mu_{q})^*/k(\mu_{q})^{*q}$ is finite, as $\mathcal{A}$ is finitely generated.

Next, the map $k^*/k^{*q}\to k(\mu_{q})^*/k(\mu_{q})^{*q}$ is none other than the restriction map $\mathrm{H}^1(k,\mu_{q})\to\mathrm{H}^1(k(\mu_{q}),\mu_{q})$, which has kernel $\mathrm{H}^1(\Gal(k(\mu_{q})/k),\mu_{q})$. 
By \cite[Prop.~9.1.6]{NSW08}, we have $\mathrm{H}^1(\Gal(k(\mu_{q})/k),\mu_{q})=0$ unless we are in the special case where $p=2$, $r\geq 2$ and $\QQ(\mu_{2^r})\cap k$ is real. In this special case, $\mathrm{H}^1(\Gal(k(\mu_{2^r})/k),\mu_{2^r})\cong\ZZ/2\ZZ$. In particular, the kernel of the natural map $k^*/k^{*q}\to k(\mu_{q})^*/k(\mu_{q})^{*q}$ is finite, and hence the finiteness of $\mathcal{X}(k,G,\mathcal{A})$ follows from the finiteness of its image in $ k(\mu_{q})^*/k(\mu_{q})^{*q}$.

	We now show that 
	$\mathcal{X}(k,G,\mathcal{A}) \subset \OO_S^* \otimes \dual{G}$;
	the rest follows from the fact that our condition
	is $S$-frobenian (see Lemma \ref{lem:frob_set}). Let $x \in k^*$ be such that its image in $k^*/k^{*q}$ is in $\mathcal{X}(k,G,\mathcal{A})$. By the argument above, the image of $x$ in $ k(\mu_{q})^*/k(\mu_{q})^{*q}$ is in $(k(\mu_{q})^*\cap K^{*q})/k(\mu_{q})^{*q}$. In particular, $x=y^{q}$ for some $y\in K^*$.
	By our assumptions in \S\ref{sec:S} that $\mathcal{A}\subset \mathcal{O}_S^*$ and that $S$ includes all primes dividing $|G|$, the extension $K/k$ is unramified at all $v\notin S$. Therefore, for all $v\notin S$, the valuation $\ord_v(x)=\ord_v(y^{q})$ is divisible by $q$.
	Consequently, the fractional ideal $x\mathcal{O}_S$ is the $q$th power of some fractional ideal $I$ of $\mathcal{O}_S$. By our assumption in \S\ref{sec:S} that $\mathcal{O}_S$ has trivial class group, $I=z\mathcal{O}_S$ for some $z\in k^*$. Therefore, $x=uz^{q}$ for some $u\in\mathcal{O}_S^*$. This completes the proof.
\end{proof}

\begin{lemma} \label{lem:leading_singularity}
	Let $x \in \OO_S^* \otimes \dual{G}$. Then $\varpi(k,G,\mathcal{A},x) = \varpi(k,G,\mathcal{A})$ if and only if $x \in \mathcal{X}(k,G,\mathcal{A})$. 
	
	Moreover
	$\widehat{f}_{\Lambda_v,G}(x_v;1) = \widehat{f}_{\Lambda_v,G}(1;1)$
	for $x\in \mathcal{X}(k,G,\mathcal{A})$ and $v \notin S$.
\end{lemma}
\begin{proof}
	Let $x \in \mathcal{X}(k,G,\mathcal{A})$. It follows 
	from the definition, $S$-frobeniality, Corollary \ref{cor:is_frob},
	and Lemma \ref{lem:varpi_def} that
	$\varpi(k,G,\mathcal{A},x) = \varpi(k,G,\mathcal{A})$. 
	The equality of Fourier transforms follows from  Lemma \ref{lem:local_good}
	and Lemma \ref{lem:finitely_many}.

	So assume that $x \notin \mathcal{X}(k,G,\mathcal{A})$.
	Let $v \notin S$ be such that	
	$x_v\notin \mathcal{A}_v \otimes \dual{G}$. Then
	$$-1 < 	|\Hom(\FF_v^*/(\mathcal{A} \bmod v), G)| - 1$$
	as this group always contains the trivial homomorphism.
	The result now follows from the fact that the function in
	Corollary \ref{cor:is_frob} is $S$-frobenian.
\end{proof}

These lemmas show that the leading singularity comes from
finitely many terms which are \emph{independent} of $S$ and our choice of local conditions for $v \in S$. This makes applications much easier when one is varying $S$ (we require such applications for the proof of Theorem \ref{thm:HNP_rare}).

\begin{theorem} \label{thm:leading_constant}
	Retain the assumptions of Theorem \ref{thm:main} and the additional
	assumptions on the finite set of places $S$ from \S \ref{sec:S}.
	Let $\mathcal{X}(k,G,\mathcal{A})$ be as in~Lemma \ref{lem:finitely_many}
	and let $S_{\mathrm{f}}$ be the set of non-archimedean places in $S$. Write $\varpi=\varpi(k,G,\mathcal{A})$.
	Then
	\begin{align*}
	c_{k,G,\Lambda} =& \frac{(\mathrm{Res}_{s=1} \zeta_k(s))^{\varpi}}{\Gamma(\varpi)|\OO_k^* \otimes G||G|^{|S_{\mathrm{f}}|}}
	\prod_{v \notin S} \left(\sum_{\substack{
	\chi_v \in \Hom(\OO_v^*, G) \\ \mathcal{A}_v \subset \Ker \chi_v}} \frac{1}{\Phi_v(\chi_v)}\right) \zeta_{k,v}(1)^{-\varpi}	 \\
	 & \,\,\,  \times 
 	 \left(\sum_{ \substack{\chi \in \Hom\bigl(\prod_{v \in S}k_v^*,G\bigr) \\
 	 \chi_v \in \Lambda_v \forall v \in S}} \frac{1}{\prod_{v \in S}\Phi_v(\chi_v)\zeta_{k,v}(1)^{\varpi}}\sum_{x \in \mathcal{X}(k,G,\mathcal{A})} 
\prod_{v \in S} \langle \chi_v, x_v \rangle\right) ,
	\end{align*}
	where the product over $v \notin S$ is non-zero.
      \end{theorem}

      \begin{proof}
        From Proposition \ref{prop:Fourier_properties}, Proposition \ref{prop:asymptotic}, and Lemma \ref{lem:leading_singularity},
        we get the leading constant
	$$c_{k,G,\Lambda} = \frac{(\mathrm{Res}_{s=1}\zeta_k(s))^{\varpi}}{\Gamma(\varpi)|\OO_k^* \otimes G|}
	\sum_{x \in \mathcal{X}(k,G,\mathcal{A})}\prod_{v}\frac{\widehat{f}_{\Lambda_v,G}(x_v;1)}{\zeta_{k,v}(1)^{\varpi}}.$$
We have 
	$\widehat{f}_{\Lambda_v,G}(x_v;1) = \widehat{f}_{\Lambda_v,G}(1;1)$
	for $x\in \mathcal{X}(k,G,\mathcal{A})$ and $v \notin S$ by Lemma \ref{lem:leading_singularity},
	and these
	factors are non-zero by Lemma \ref{lem:holomorphic}.
	The explicit expressions for $v \notin S$ follow from Lemma \ref{lem:O_v_sum}.
	For $v \in S$, we simply apply
	directly the definition of the local Fourier transforms
	from \S\ref{sec:Fourier_def} (see \eqref{eq:localfourier} for a formula in the non-archimedean case) and change the order of summation.
\end{proof}

Note that the expression for $c_{k,G,\Lambda}$ is independent of $S$, for any $S$ which satisfies the assumptions of \S \ref{sec:S}.

\begin{remark}
	In the special case $\mathcal{A} = \{1\}$, our constant  agrees with
	the constant which Wood obtains in \cite[Thm.~3.1]{Woo10},
	up to the factor $(\mathrm{Res}_{s=1} \zeta_k(s))^{\varpi(k,G,\mathcal{A})}$.
	This factor is missing from Wood's paper: in the proof of \cite[Thm.~3.1]{Woo10}, she mistakenly uses the equality $\lim_{s \to 1}(s-1)\zeta_K(s) = 1$, which holds for $K=\mathbb{Q}$ but does not hold in general (the residue is given by the analytic class number formula). Thus the right-hand side of \cite[Thm.~3.1]{Woo10} should contain an additional factor of $(\mathrm{Res}_{s=1} \zeta_K(s))^{w_{K,C}}$.	
\end{remark}

\begin{remark} Let $\widehat{G}=\Hom(\dual{G},\mathbb{G}_m) $ denote the Cartier dual of $\dual{G}$. Then $\Sha(k, \widehat{G})=\Ker(k^*\otimes\dual{G}\to\Adele^*\otimes\dual{G})$.
	An examination of the proof of Lemma~\ref{lem:finitely_many} gives
	the bounds
	$$|\Sha(k, \widehat{G}) \cdot ( \mathcal{A} \otimes \dual{G})|
	\leq |\mathcal{X}(k,G,\mathcal{A})| \leq 
	|2\dual{G}/4\dual{G}| |\mathcal{A} \otimes \dual{G}|.$$
	The following examples show that either bound can be sharp.
	
	For the lower bound, take $k=\QQ$,
	$\mathcal{A} = \{ 1\}$ and $\dual{G} = \ZZ/4\ZZ$. Then $\mathcal{A}\otimes\dual{G}=1$, $\cX(k,G,\{1\})=\Sha(k,\widehat{G})=0$, and $|2\dual{G}/4\dual{G}|=2$.
	
	For the upper bound, take $k=\QQ$,
	$\mathcal{A} = \{\pm 1\}$ and $\dual{G} = \ZZ/4\ZZ$.
	Then one checks that $\mathcal{X}(\QQ,\ZZ/4\ZZ,\{ \pm 1\}) = \langle \pm 1,
	\pm 4\rangle$, despite the fact that $\Sha(\QQ, \mu_4) = 0$.
	
	 An example where both bounds coincide is given by taking $\mathcal{A} =\{1\}$ and $\dual{G}=\ZZ/2\ZZ$. One easily sees that in this case $\mathcal{X}(k,G,\{1\})$ is trivial.
\end{remark}

\subsection{Positivity of the leading constant}
To finish the proof of Theorem~\ref{thm:main}, we need to show that $c_{k,G,\Lambda} > 0$ if there exists some sub-$G$-extension which realises all the given local conditions. It suffices to consider the contributions from $v \in S$ to the explicit expression given in Theorem \ref{thm:leading_constant}, as the factors at $v \notin S$ are clearly non-zero.
By character orthogonality we have
$$\sum_{x \in \mathcal{X}(k,G,\mathcal{A})} \prod_{v \in S} 
	\langle \chi_v, x_v \rangle = 
\begin{cases}
	|\mathcal{X}(k,G,\mathcal{A})| & \text{if }\prod_{v \in S} \chi_v \text{ is trivial on } \mathcal{X}(k,G,\mathcal{A}), \\
	0 & \text{otherwise}.
\end{cases}$$
In particular, this sum is non-negative for all $\chi\in \Hom(\prod_{v \in S}k_v^*,G)$. Hence, it suffices to show the existence of some $\chi$ such that this sum is non-zero. However, we have assumed the existence of a sub-$G$-extension $\varphi$ which realises all the local conditions. Let $\psi: \Adele^*/k^* \to G$ be the associated homomorphism coming from class field theory. Note that $\prod_v \langle \psi_v, x_v \rangle = 1$ for all $x \in k^* \otimes \dual{G}$, hence
$$\prod_{v \in S} \langle \psi_v, x_v \rangle = \prod_{v \notin S} 
\frac{1}{\langle \psi_v, x_v \rangle}.$$
It therefore suffices to show that 
\begin{equation} \label{eqn:psi_x}
	\langle \psi_v, x_v \rangle = 1 \quad \text{ for all } v \notin S
	\text{ and all } x \in \mathcal{X}(k,G,\mathcal{A}).
\end{equation}
However, for $x \in \mathcal{X}(k,G,\mathcal{A})$ we have
$x_v \in \mathcal{A}_v \otimes \dual{G}$
for all $v \notin S$, by Lemma~\ref{lem:finitely_many}. Moreover,
by assumption every element of $\mathcal{A}$ is a local norm from $K_\varphi$
for all $v \notin S$, thus 
$\mathcal{A}_v \subset \Ker \psi_v$ for all $v \notin S$ by Lemma \ref{lem:kernel}.
The claim \eqref{eqn:psi_x} follows, which completes the proof of Theorem \ref{thm:main}. \qed

\section{Proof of results} \label{sec:proofs}

We now apply Theorem \ref{thm:main} in various ways to prove the results from the introduction.

\subsection{Asymptotic formula for everywhere local norms}
We first derive an asymptotic formula for $N_{\mathrm{loc}}(k,G,\mathcal{A},B)$ (see \eqref{def:counting_functions}) using Theorem \ref{thm:main}.

\begin{theorem} \label{thm:loc}
	We have
	$$
	N_{\mathrm{loc}}(k,G,\mathcal{A},B) = c_{k,G,\mathcal{A},\mathrm{loc}}B(\log B)^{\varpi(k,G,\mathcal{A}) -1}
	+  O(B(\log B)^{\varpi(k,G,\mathcal{A})-1-\delta})$$
	as $B\to\infty,$
	for some $ c_{k,G,\mathcal{A},\mathrm{loc}} > 0$ and some $\delta=\delta(k,G,\mathcal{A}) >0$.
\end{theorem}
\begin{proof}
	For all $v\in\Omega_k$, let $\Lambda_v$ be the set of sub-$G$-extensions of 
	$k_v$ corresponding	to those extensions $L/k_v$ for which 
	every element of $\mathcal{A}$ 
	is a local norm from $L/k_v$. Thus, in this setting $\Lambda=(\Lambda_v)_{v\in\Omega_k}$ is determined by $\cA$. We clearly have
	$N_{\mathrm{loc}}(k,G,\mathcal{A},B) = N(k,G,\Lambda,B)$.
	It therefore suffices to show that the leading constant in 
	Theorem \ref{thm:main} is positive. To do so,
	we need to exhibit some \emph{sub}-$G$-extension
	of $k$ for which every element of $\mathcal{A}$
	is everywhere locally a norm. However, the trivial
	extension $k/k$  is such an extension.
\end{proof}

\subsection{Proof of Theorem \ref{thm:HNP_rare}}
As cyclic extensions always satisfy the Hasse norm principle, we may assume that $G$ is non-cyclic.
We use the following criterion for failure of the Hasse norm principle in the abelian setting, which was originally pointed out to us by Melanie Matchett Wood. (We use the notation from \S\ref{sec:statement_main}.)



\begin{proposition} \label{prop:HNP}
	Let $\varphi$ be a $G$-extension of $k$. Then $\varphi$
	fails the Hasse norm principle if and only if
	there exists a proper subgroup $\Upsilon \subset \wedge^2 (G)$ that contains the image of the natural map
	$$\prod_v\wedge^2(\Im \varphi_v)\to \wedge^2(G).$$
\end{proposition}

\begin{proof}
	Let $K$ be the number field determined by $\varphi$.
	Recall that the failure of the Hasse norm principle is measured by the
	Tate--Shafarevich group
	$$ \Sha(k,\Res_{K/k}^1\Gm) := \Ker\Bigl( \HH^1(k,\Res_{K/k}^1\Gm) \to 
	\prod_{v}\HH^1(k_v,\Res_{K/k}^1\Gm)\Bigr),$$
	where $\Res_{K/k}^1\Gm$ denotes the associated norm $1$ torus, see \cite[\S 6.3]{PR94}.
	This group is finite by \cite[Prop.~6.9]{PR94}. As $K/k$ is Galois, a theorem of Tate \cite[Thm.~6.11]{PR94} (see also \cite[Ex.~5.6]{San81}) implies that there
	is an exact sequence
	$$0 \to \Hom(\Sha(k,\Res_{K/k}^1\Gm), \QQ/\ZZ) \to \HH^3(G, \ZZ) \to 
	\prod_{v } \HH^3(\Im \varphi_v,\ZZ).$$
	However, as $G$ is abelian, we have a well-known canonical isomorphism
	$$\HH^3(G,\ZZ) \cong \Hom(\wedge^2(G), \QQ/\ZZ)$$
	(see e.g.~\cite[Lem.~6.4]{HNP}). Using this and applying $\Hom(\cdot,\QQ/\ZZ)$,
	we therefore obtain the exact sequence
	\begin{equation} \label{eqn:Sha}
	\prod_{v} \wedge^2(\Im \varphi_v) \to  \wedge^2(G) 
	\to \Sha(k,\Res_{k_\varphi/k}^1\Gm) \to 0.
	\end{equation}
	Thus, failure of the Hasse norm principle is equivalent 
	to the first map in \eqref{eqn:Sha} failing to be surjective. 
\end{proof}

Therefore, to prove Theorem \ref{thm:HNP_rare}, it suffices to show the following.

\begin{theorem} \label{thm:Upsilon}
	Let $\Upsilon \subset \wedge^2(G)$ be a proper subgroup. Then
	$$
	\lim_{B \to \infty}
	\frac{\#\left\{ \varphi \in \gextk \,: 
	\Phi(\varphi) \leq B, \mathcal{A} \subset \Norm_{K_\varphi/k} \Adele_{K_\varphi}^*, 
	\wedge^2(\Im \varphi_v)\subset \Upsilon \, \forall v 
	\right\}}{N_{\mathrm{loc}}(k,G,\mathcal{A},B)}  =0.
	$$
\end{theorem}
 Note that in Theorem~\ref{thm:Upsilon}, and henceforth, we abuse notation by writing $\wedge^2(\Im \varphi_v)\subset \Upsilon$ to mean that the image of the natural map $\wedge^2(\Im \varphi_v)\to \wedge^2(G)$ is contained in $\Upsilon$, despite the fact that this map is not injective in general.

We prove Theorem~\ref{thm:Upsilon} via an application of Theorem \ref{thm:main}. Note, however, that one cannot
apply Theorem \ref{thm:main} directly, as the local conditions imposed at the infinitely many places will not be compatible with the assumptions of Theorem \ref{thm:main}. We therefore apply  Theorem \ref{thm:main} to a suitable finite set of places, which we then allow to increase.

\subsubsection{Proof of Theorem \ref{thm:Upsilon}}
Let $S_0$ be a finite set of places of $k$ satisfying the conditions of \S \ref{sec:S}, which we consider as being fixed.
Let $T$ be a finite set of places of $k$ which is disjoint from $S_0$. Eventually, we will consider what happens as $T$ increases. Let $S=S_0\cup T$.

We consider the local conditions $\Lambda_v$ given by
\begin{align*}
\{\varphi_v\in\Hom(\Gal(\bar{k}_v/{k_v}), G) : \mathcal{A}_v \subset \Norm_{K_{\varphi_v}/k_v}(K_{\varphi_v}^*)\}, &\ v\notin T;\\
\{\varphi_v\in\Hom(\Gal(\bar{k}_v/{k_v}),G) : \mathcal{A}_v \subset \Norm_{K_{\varphi_v}/k_v}(K_{\varphi_v}^*),
\wedge^2(\Im \varphi_v)\subset \Upsilon\}, & \ v\in T.
\end{align*}
We denote the collection of such conditions by $\Lambda_T$.
Note that we clearly have
$$	\frac{\#\left\{ \varphi \in \gextk \,: 
\begin{array}{ll}
	&\Phi(\varphi) \leq B, \mathcal{A} \subset \Norm_{K_\varphi/k} \Adele_{K_\varphi}^*, \\
	&\wedge^2(\Im \varphi_v)\subset \Upsilon \, \forall v 
\end{array}
	\right\}}{N_{\mathrm{loc}}(k,G,\mathcal{A},B)} 
\leq \frac{N(k,G,\Lambda_T,B)}{N_{\mathrm{loc}}(k,G,\mathcal{A},B)}$$
for all $B$.
Applying Theorem \ref{thm:main} gives
$$
	\lim_{B \to \infty}\frac{N(k,G,\Lambda_T,B)}{N_{\mathrm{loc}}(k,G,\mathcal{A},B)}
	=\frac{c_{k,G,\Lambda_T}}{c_{k,G,\mathcal{A},\mathrm{loc}}}
$$
where $c_{k,G,\mathcal{A},\mathrm{loc}}>0$ by Theorem \ref{thm:loc}.
To prove Theorem \ref{thm:Upsilon} it therefore suffices to show that
\begin{equation} \label{eqn:lim_equals_0}
\lim_{S_0 \cup T \to \Val}\frac{c_{k,G,\Lambda_T}}{c_{k,G,\mathcal{A},\mathrm{loc}}} = 0
\end{equation}
where as explained we consider $S_0$ as fixed and $T$ as increasing and disjoint from $S_0$.
We do this using the explicit expression for the leading constant
given in Theorem \ref{thm:leading_constant}. We let $e$ be the exponent of $G$.  We require the following elementary observation.

\begin{lemma}	\label{lem:Rachel}
	Let $\alpha \in k^{*}$. If $v$ is such that $\alpha \in k_v^{*e}$, then
	$\alpha$ is a local norm at $v$ from every sub-$G$-extension of $k$.
\end{lemma}
\begin{proof}
	Let $K$ be an extension of $k$ with Galois group isomorphic to a subgroup of $G$
	and $v$ a place of $k$ such that $\alpha \in k_v^{*e}$.
	Let $K_v$ be the
	completion of $k$ at a choice of place of $K$ above $v$. Then local class field
	theory yields 
	$$k_v^*/\Norm_{K_v/k_v}K_v^*\cong \Gal(K_v/k_v)\hookrightarrow G.$$
	Now $G$ has exponent $e$, whereby the group $k_v^*/\Norm_{K_v/k_v}K_v^*$
	has exponent dividing $e$. It follows that an $e$th power in $k_v^*$ is a local norm.
\end{proof}

We now obtain the following bounds.

\begin{lemma} \label{lem:constant_calc}
  Let $k_e = k(\mu_e, \sqrt[e]{\mathcal{A}})$. Then 
	$$\frac{c_{k,G,\Lambda_T}}{c_{k,G,\mathcal{A},\mathrm{loc}}}
	\leq \prod_{\substack{v \in T \\ v \text{ completely
	split in }k_e/k}} \frac{\sum_{\substack{\chi_v \in \Hom(k_v^*, G) \\ \wedge^2(\im\chi_v)\subset \Upsilon}} \frac{1}{\Phi_v(\chi_v)}}
	{\sum_{\substack{\chi_v \in \Hom(k_v^*, G)}} \frac{1}{\Phi_v(\chi_v)}}. $$
\end{lemma}
\begin{proof}
	The factors in Theorem \ref{thm:leading_constant}
	cancel out in the quotient 
	$c_{k,G,\Lambda_T}/c_{k,G,\mathcal{A},\mathrm{loc}}$,
	except those at places $v \in S$.
	By Lemma~\ref{lem:finitely_many}, we have
	\[\mathcal{X}(k,G,\mathcal{A}) = \{ x \in \OO_{S_0}^* \otimes \dual{G} :
	x_v \in \mathcal{A}_v \otimes \dual{G} \text{ for all } v \notin S_0\}\]
	(this statement holds for any set of places satisfying the assumptions of \S \ref{sec:S}).
	Moreover, for $v \in T$ any element of $\mathcal{A}_v$
	is a local norm at $v$ by our choice of $\Lambda_v$;
	it follows that $\langle \chi_v, x_v \rangle = 1$ for $\chi_v \in \Lambda_v$
	as in Theorem \ref{thm:leading_constant}, hence
	$$	\sum_{x \in \mathcal{X}(k,G,\mathcal{A})} 
	\prod_{v \in S} \langle \chi_v, x_v \rangle
	= \sum_{x \in \mathcal{X}(k,G,\mathcal{A})} 
\prod_{v \in S_0} \langle \chi_v, x_v \rangle.$$
        Therefore, we can split off Euler factors for all $v\in T$ from the term involving $S$, while the remaining sum over $\Hom(\prod_{v\in S_0}k_v^*,G)$ is the same in $c_{k,G,\Lambda_T}$ and $c_{k,G,\mathcal{A},\mathrm{loc}}$. We have obtained the equality
        \begin{equation*}
          \frac{c_{k,G,\Lambda_T}}{c_{k,G,\mathcal{A},\mathrm{loc}}}=\prod_{v\in T} \frac{\sum_{\chi_v\in\Lambda_v}\frac{1}{\Phi_v(\chi_v)}}{\sum_{\substack{\chi_v\in\Hom(k_v^*,G)\\\mathcal{A}_v\subset\Ker\chi_v}}\frac{1}{\Phi_v(\chi_v)}}.
          \end{equation*}
	The quotient of each local factor is at most $1$,
	so to obtain an upper bound we may just consider
	those places $v \in T$ which are completely split in $k_e/k$.
	For such places every element of $\mathcal{A}$
	is an $e$th power in $k_v^*$, hence the
	condition that they are local norms is automatic by
	Lemma \ref{lem:Rachel}. The result follows.
\end{proof}

We will make use of the following fact from \cite[Lem.~6.9]{HNP}. Here, we use the term \emph{bicyclic} for a non-cyclic group that is a direct sum of two cyclic groups.

\begin{lemma} \label{lem:G_i}
	Let $G$ be a finite abelian non-cyclic group. Then there exists a finite collection
	of bicyclic subgroups $G_i \subset G$ for $i \in I$ such that the natural map
	$$\bigoplus_{i \in I} \exterior (G_i) \to \exterior (G)$$
	is an isomorphism.
\end{lemma}

As $\Upsilon \subset \wedge^2(G)$ is a proper subgroup, there exists some $i$ such that $\wedge^2 (G_i) \not \subset \Upsilon$. Fix this $i$ and write $G_i \cong \ZZ/n\ZZ \times \ZZ/m\ZZ$ where $n,m \mid e$. Let $v \in T$ be a place of $k$ which is completely split in the extension $k(\mu_e, \sqrt[e]{\mathcal{A}})$. There exists a $G_i$-extension of $k_v$: simply adjoin an $n$th root of a uniformiser to the unique unramified extension of $k_v$ of degree $m$. Thus, by local class field theory, there exists $\chi_v \in \Hom(k_v^*,G)$ such that $\Im \chi_v = G_i$. In particular we have $\wedge^2(\Im \chi_v) \not \subset \Upsilon$. 
For such places $v$ we find that
\begin{align*}
	& \# \{ \chi_v \in \Hom(k_v^*,G) : 
	\wedge^2(\Im \chi_v)\subset \Upsilon,
	\chi_v \text{ ramified}\}  \\
	& < \# \{ \chi_v \in \Hom(k_v^*,G) :
	\chi_v \text{ ramified}\}.
\end{align*}
Let $a_v= \# \{ \chi_v \in \Hom(k_v^*,G) : \chi_v \text{ ramified}\}$. Recall that tamely ramified $\chi_v$ have conductor $q_v$ and there are $|G|$ unramified $G$-characters.
Using Lemma \ref{lem:constant_calc}, it follows that 
\begin{align*}
\frac{c_{k,G,\Lambda_T}}{c_{k,G,\mathcal{A},\mathrm{loc}}}&\leq \prod_{\substack{v \in T  \\ v \text{ completely
	split in }k_e}} \frac{|G|+\frac{a_v-1}{q_v}+O\left( \frac{1}{q_v^2}\right)}{|G|+\frac{a_v}{q_v}+O\left( \frac{1}{q_v^2}\right)}\\
	&= \prod_{\substack{v \in T  \\ v \text{ completely
	split in }k_e}} \left(1 - \frac{1}{|G|q_v} + O\left( \frac{1}{q_v^2}\right)\right)
\end{align*}
However this diverges to $0$ as $S_0 \cup T \to \Omega_k$ since
$$-\sum_{v \text{ completely 	split in }k_e} \frac{1}{q_v}$$
diverges by  the Chebotarev density theorem. This proves \eqref{eqn:lim_equals_0} and completes the proof of Theorem \ref{thm:Upsilon}, hence the proof of Theorem \ref{thm:HNP_rare}. \qed

\begin{remark}
	Lemma \ref{lem:G_i} is the first part of the statement of \cite[Lem.~6.9]{HNP}.
	Unfortunately the second part of the statement \cite[Lem.~6.9]{HNP} is false
	(this claims that if the exponent of $\wedge^2(G)$ divides a prime $p$, then all
	the $G_i$ may chosen isomorphic to $(\ZZ/p\ZZ)^2$). A counterexample is given by
	the group $G = \ZZ/2\ZZ \times \ZZ/4\ZZ$ and the subgroup $G_1 = \ZZ/2\ZZ \times \ZZ/2\ZZ$;
	here the induced map
	$$\ZZ/2\ZZ = \wedge^2 (G_1) \to \wedge^2 (G) = \ZZ/2\ZZ$$
	is trivial. This mistake in \cite[Lem.~6.9]{HNP} has various consequences for \cite{HNP}
	which will be addressed in a forthcoming corrigendum.
\end{remark}

\subsection{Proof of Theorem \ref{thm:global}}
Follows from Theorems~\ref{thm:loc} and \ref{thm:HNP_rare}.  \qed

\subsection{Proof of Theorem \ref{thm:existence}}
Follows immediately from Theorem \ref{thm:global}. \qed

\subsection{Proof of Theorem \ref{thm:norm_rare}}
The implication (3)$\Rightarrow$(1) in Theorem \ref{thm:norm_rare} follows
from Lemma \ref{lem:Rachel}, Theorem~\ref{thm:main} and Theorem \ref{thm:HNP_rare}, as clearly 
$\varpi(k,G,\mathcal{A}) = \varpi(k,G,\{1\})$ in this case (we are only imposing finitely
many local conditions).
The implication (1)$\Rightarrow$(2) is clear from Definition \ref{def:varpi_intro} and Theorem~\ref{thm:global}. For the remaining implication (2)$\Rightarrow$(3), we note that (2) clearly implies that
\begin{equation}
  \label{eq:varpi_local_simple}
  \mathcal{A}_v \subset k_v^{*d} \text{ for all } d \mid e \text{ and all } v\nmid\infty \text{ with } q_v \equiv 1 \bmod d.
\end{equation}
Moreover, we have the following elementary observation.
\begin{lemma} \label{lem:local_power}
	Let $e \in \ZZ_{\geq 1}$, let $\alpha \in k^*$, and let $v$ be a place of $k$ such that
	$e,\alpha \in \OO_v^*$.
	Let $d = \gcd(e, q_v-1)$.
	If $\alpha \in k_v^{*d}$ then $\alpha \in k_v^{*e}$.
\end{lemma}
\begin{proof}
	As $\alpha \in k_v^{*d}$ and $\alpha$ is a unit, its image in the residue field lies in
	$\FF_v^{*d}$. However, as $d = \gcd(e, q_v-1)$, we have $\FF_v^{*d}=\FF_v^{*e}$.
	The result therefore follows from Hensel's lemma. 
\end{proof}

Hence, the remaining implication (2)$\Rightarrow$(3) in Theorem \ref{thm:norm_rare} follows immediately from  \eqref{eq:varpi_local_simple}
and Lemma \ref{lem:local_power}. \qed

\subsection{Proof of Corollary \ref{cor:norm_rare}} \label{sec:cor_norm_rare}
Let $\alpha \in \mathcal{A}$ and consider $\alpha \beta^e$ for some $\beta \in k^*$.
By Lemma \ref{lem:Rachel}, we see that $\beta^e$ is a norm everywhere locally from all $G$-extensions of $k$. It follows that $\alpha \beta^e$ is a norm everywhere locally from a given $G$-extension if and only if $\alpha$ is a norm everywhere locally. Part $(i)$ now follows from Theorem~\ref{thm:HNP_rare} and Theorem \ref{thm:loc}. Part $(ii)$ also follows from Lemma \ref{lem:Rachel} and Theorem \ref{thm:HNP_rare}.

For $(iii)$ and $(iv)$, we use the $\omega$-version of Tate--Shafarevich groups. Namely, for a finite abelian group scheme $M$ over $k$ we let
$$\Sha_\omega(k,M) = \{ c \in \HH^1(k,M) : c_v = 0 \in \HH^1(k_v,M) \text{ for all but finitely many }v \}.$$
By Kummer theory we have $ \HH^1(L, \mu_e)= L^*/L^{*e}$ for any field $L$ of characteristic $0$. Therefore Part $(iii)$ of Theorem \ref{thm:norm_rare} is equivalent to
$$\mathcal{A}k^{*e} \subset \Sha_\omega(k,\mu_e).$$
The key observation is now the following.

\begin{lemma} \label{lem:Sha}
	Let $k$ be a number field, let $e \in \ZZ_{\geq 1}$
	and let $2^r$ be the largest power of $2$ dividing $e$.
	Then $\Sha_\omega(k,\mu_e) =0$,
	unless the extension $k(\mu_{2^r})/k$
	is non-cyclic, where we have $\Sha_\omega(k,\mu_e) \cong \ZZ/2\ZZ$.
\end{lemma}
\begin{proof}
	Follows immediately from \cite[Thm.~9.1.11]{NSW08}.
\end{proof}

The remaining parts of Corollary \ref{cor:norm_rare} now follow from  Lemma~\ref{lem:Sha}. \qed 

\begin{remark}
Even though we have $N_{\mathrm{loc}}(k,G,\mathcal{A},B) = N_{\mathrm{loc}}(k,G,\mathcal{A}\langle\beta^e\rangle,B)$ and $N_{\mathrm{glob}}(k,G,\mathcal{A},B) \sim N_{\mathrm{glob}}(k,G,\mathcal{A}\langle\beta^e\rangle,B)$, we can still have $N_{\mathrm{glob}}(k,G,\mathcal{A},B) \neq  N_{\mathrm{glob}}(k,G,\mathcal{A}\langle\beta^e\rangle,B)$.
For example, take $k=\QQ, G = (\ZZ/2\ZZ)^2, \mathcal{A} = \{1\},$ and $\beta = 5$ (see Example \ref{examples}(4)).
\end{remark}

\subsection{Proof of Theorem \ref{thm:positivenot1}}
The implication (2) $\Rightarrow $ (1) is self-evident. So suppose that $\lim_{B \to \infty} \frac{N_{\mathrm{glob}}(k,G,\mathcal{A},B)}{N(k,G,B)} >0$. Then by Theorem~\ref{thm:norm_rare} there exists a cofinite set of places $T \subset \Omega_k$ such that $\mathcal{A} \subset k_v^{*e}$  for all $v\in T$. By \cite[Thm.~9.1.11]{NSW08},
\[\Ker\Bigl(k^*/k^{*e}\to \prod_{v\in T} k_v^*/k_v^{*e}\Bigr)=\Ker\Bigl(k^*/k^{*e}\to \prod_{v\in T\cup\{v\nmid 2\}} k_v^*/k_v^{*e}\Bigr)\]
so we may assume that $T$ contains all $v\nmid 2$.
Let $\mathfrak{p}$ be the unique prime of $k$ lying above $2$. Let $\chi:\Gal(\bar{k}/k)\to G$ be a $G$-extension and let $\alpha\in\cA$. Then at all places $v\neq \mathfrak{p}$, the cyclic algebra $(\chi,\alpha)$ over $k$ has local invariant zero, because $\alpha$ is a local norm at $v$ by Lemma \ref{lem:Rachel}. Now the Albert--Brauer--Hasse--Noether Theorem \cite[Thm.~8.1.17]{NSW08} shows that $(\chi,\alpha)$ has local invariant zero at $\mathfrak{p}$, meaning that $\alpha$ is also a local norm at $\mathfrak{p}$. Therefore, all elements of $\mathcal{A}$ are everywhere local norms from all $G$-extensions of $k$. But $G$ is cyclic, hence every $G$-extension satisfies the Hasse norm principle; (2) now follows. \qed

\subsection{Variants of Theorems \ref{thm:existence} and \ref{thm:global}}

We finish with some variants of our results, which allow one to impose local conditions at finitely many places. Our first result is a variant of Theorem \ref{thm:global}, and follows immediately from Theorem \ref{thm:main} and Theorem \ref{thm:HNP_rare}.

\begin{corollary} \label{cor:existence_S}
	Retain the assumptions of Theorem \ref{thm:main}. Assume further
	that every element of $\mathcal{A}$ is a local norm from every
	extension in $\Lambda_v$ for all $v$, and that there exists a 
	sub-$G$-extension of $k$
	which realises the given local conditions for all places $v$.
	Then
	\begin{align*}
		&\#\{\varphi\in\gextk:\, \Phi(\varphi)\leq B, \, \varphi_v \in \Lambda_v \,\forall v, \, \mathcal{A} \subset \Norm_{K_\varphi/k} K_\varphi^* \} \\
		&= c_{k,G,\Lambda}B(\log B)^{\varpi(k,G,\mathcal{A}) -1}(1 +o(1)),
	\end{align*}
	for some leading constant $ c_{k,G,\Lambda} > 0$.
\end{corollary}

From this we immediately obtain the following strengthening of Theorem \ref{thm:existence}.

\begin{corollary} \label{cor:existence}
		Let $k$ be a number field, $S$ a finite set of places of $k$, $G$ a finite abelian group, and $\mathcal{A} \subset k^*$ a finitely generated subgroup. Let $\psi$ be a sub-$G$-extension of $k$ such that every element of $\mathcal{A}$ is everywhere locally a norm from $K_\psi$.  There exists a $G$-extension $\varphi$ of $k$ such that every element of $\mathcal{A}$ is a global norm from $K_\varphi$ and such that $\varphi_v = \psi_v$ for all $v \in S$.
\end{corollary}

\begin{remark} \label{rem:Olivier}
Taking $\psi$ to correspond to the trivial extension $k/k$, we find the existence of an extension $K/k$ with Galois group $G$ such that every element of $\mathcal{A}$ is a norm from $K$ and such that $K$ is completely split at all places of $S$.
\end{remark}

\appendix

\section{An algebro-geometric point of view on Theorem~\ref{thm:existence}}

\begin{center}
{\sc by Yonatan Harpaz and Olivier Wittenberg }
\end{center}

\newcommand{\SL}{\mathrm{SL}}
\newcommand{\Z}{\ZZ}
\newcommand{\A}{\mathbb A}

\medskip
We give, in this appendix, an algebro-geometric proof of Theorem~\ref{thm:existence}, based
on a combination of the descent and fibration methods in the formulation they are given
in~\cite{HW18}. The main argument is described in \textsection\textsection\ref{app:subsecstatement}--\ref{app:subsecfib}. In~\S\ref{app:subsecvertical} we show that the refinement of Theorem~\ref{thm:existence} formulated in Corollary~\ref{cor:existence} can also be deduced in this manner by proving a certain verticality result on the Brauer groups of the varieties in question. This verticality uses in an essential way the fact that $G$ is abelian. In \S\ref{app:nonabelien} we show that when $G$ is not abelian, the statement of Corollary~\ref{cor:existence} is \emph{false}, by constructing a counterexample in the form of an explicit $2$-group. Nonetheless, as we show in the upcoming work~\cite{supersolvable}, Theorem~\ref{thm:existence} does hold for $2$-groups (and more generally for
nilpotent groups, even for supersolvable groups).

Let us fix, for the whole of \textsection\textsection\ref{app:subsecstatement}--\ref{app:subsecvertical},
a finite abelian group~$G$,
a field~$k$ of characteristic~$0$,
a finite collection $\alpha_1,\dots,\alpha_m \in k^*$
and an algebraic closure~$\bar k$ of~$k$.
In \textsection\textsection\ref{app:subsecstatement}--\ref{app:subsecfib},
we assume that~$k$ is a number field.

\subsection{Statements}
\label{app:subsecstatement}

Let us
choose an embedding $G \hookrightarrow \SL_n(k)$ for some~$n\geq 1$.
Let~$\SL_n$ and~$\Gm$ implicitly denote the corresponding algebraic groups over~$k$.
For any $\alpha \in k^*$,
let $T^\alpha \subset \prod_{g \in G} \Gm$ denote the
subvariety whose $\bar k$\nobreakdash-points are the maps $t:G \to \bar k^*$
such that $\prod_{g \in G}t(g)=\alpha$.  Thus~$T^\alpha$ is a (trivial)
torsor under the (trivial) torus~$T^1$.

Let $Y = \SL_n \times T^{\alpha_1} \times \dots \times T^{\alpha_m}$.
Let~$G$ act on~$\SL_n$ by right multiplication,
on $T^\alpha$ (for any~$\alpha$) by the right action $(t\cdot \gamma)(\gamma')=t(\gamma\gamma')$,
and on~$Y$ by the resulting diagonal right action.
As~$G$ acts freely on~$\SL_n$, it acts freely on~$Y$;
hence, letting $X = Y/G$,
the quotient map $\pi:Y \to X$ is a $G$\nobreakdash-torsor.

The fibre of~$\pi$ above any rational point of~$X$ is a $G$\nobreakdash-torsor over $\Spec(k)$,
say $\Spec(K)$, where~$K$ is an \'etale $k$\nobreakdash-algebra endowed with elements
$\beta_1,\dots,\beta_m \in K^*$ such that $N_{K/k}(\beta_i)=\alpha_i$ for all~$i$.
Namely $\beta_i$ is the restriction to the fibre in question
of
the invertible function $(s,t_1,\dots,t_m) \mapsto t_i(1)$ on~$Y$,
where~$1$ denotes the identity element of~$G$.
If the algebra~$K$ is a field, then $K/k$ is Galois with group~$G$, since it is a $G$\nobreakdash-torsor.
Thus, all we need to do, to show Theorem~\ref{thm:existence},
is to prove that there exists $x \in X(k)$ such that $\pi^{-1}(x)$ is irreducible.
Letting~$X'$ denote a smooth compactification of~$X$ (i.e.\ any smooth and proper
variety containing~$X$ as a dense open subset), we shall
in fact prove the following theorem.

\begin{theorem}
\label{app:theorem}
The set $X'(k)$ is dense in the Brauer--Manin set $X'(\A_k)^{\Br(X')}$.
\end{theorem}

We note that $X'(\A_k)^{\Br(X')}$ is non-empty since so is $X(k)$;
indeed, even $Y(k)$ is non-empty.
The desired result now follows from
Theorem~\ref{app:theorem}:

\begin{corollary}
\label{app:corollary}
The set of $x \in X(k)$ such that $\pi^{-1}(x)$ is irreducible
is dense in the (non-empty) Brauer--Manin set $X'(\A_k)^{\Br(X')}$.
\end{corollary}

\begin{proof}
This is essentially an application of a theorem of Ekedahl~\cite[Thm.~1.3]{ekedahl}
(also discussed and proved in \cite[\textsection\textsection3.5--3.6]{serretopics}).
What Ekedahl really shows in~\cite{ekedahl} is that
for any finite \'etale morphism $\pi:Y\to X$ between geometrically irreducible varieties over~$k$
and for any finite set~$S$ of places of~$k$, there exist a finite set~$S'$ of places of~$k$,
disjoint from~$S$, and a collection $(x'_v)_{v \in S'} \in \prod_{v \in S'} X(k_v)$
such that for any $x \in X(k)$ close enough to $(x'_v)_{v \in S'}$ for the product topology
on $\prod_{v \in S'} X(k_v)$,
the scheme $\pi^{-1}(x)$
is irreducible.
Theorem~\ref{app:theorem} implies Corollary~\ref{app:corollary}
in view of this statement and of the remark that
the Brauer--Manin set is open in~$X'(\A_k)$
as~$X$ is geometrically unirational
(see \cite[Remarks~2.4 (i)--(ii)]{wslc}).
\end{proof}

As we have seen,
Corollary~\ref{app:corollary}
implies
Theorem~\ref{thm:existence}.
In a less immediate way, it also implies
Corollary~\ref{cor:existence}.
Indeed, noting that any sub-$G$\nobreakdash-extension of~$k$, in the terminology
of~\textsection\ref{sec:statement_main}, arises as the fibre of the quotient
map $\SL_n \to \SL_n/G$ above a rational point of $\SL_n/G$
(see \cite[Ch.~I, \textsection5.4, Cor.~1]{Ser94} and recall that $\HH^1(k,\SL_n)$
is a singleton by Hilbert's Theorem~90),
we see that Corollary~\ref{cor:existence} follows
from combining Corollary~\ref{app:corollary}
with Proposition~\ref{app:propvertical} below.

\begin{proposition}
\label{app:propvertical}
Let $B=\SL_n/G$ and $b \in B(k)$.  Let $f:X \to B$ be the map induced by the first projection $Y \to \SL_n$.
Let~$\Omega$ denote the set of places of~$k$.
Let $(x_v)_{v\in \Omega} \in \prod_{v \in \Omega} X(k_v)$.
If $f(x_v)=b$ for all $v\in\Omega$,
then $(x_v)_{v\in\Omega} \in X'(\A_k)^{\Br(X')}$.
\end{proposition}

We shall prove Theorem~\ref{app:theorem} in
\textsection\textsection\ref{app:subsecdescent}--\ref{app:subsecfib}
and Proposition~\ref{app:propvertical}
in~\textsection\ref{app:subsecvertical}.

\subsection{Descent}
\label{app:subsecdescent}

To prove Theorem~\ref{app:theorem}, we first perform a descent, in the sense of Colliot-Th\'el\`ene and
Sansuc~\cite{ctsansuc},
to reduce ourselves to studying the
arithmetic of hopefully simpler auxiliary varieties.
If~$V$ is a variety over~$k$,
we denote by~$\bar k[V]^*$ the group of global invertible functions on $V \otimes_k \bar k$.

\begin{proposition}
\label{app:proposition}
We have $\bar k[X]^*=\bar k^*$.
\end{proposition}

\begin{proof}
We first remark that there is a canonical exact sequence of abelian groups
\begin{align}
\label{appendix:eqky}
0 \to \bar k^* \to \bar k[Y]^* \to (\Z[G]/\Z)^m \to 0
\end{align}
whose arrows are equivariant with respect to the actions of $\Gal(\bar k/k)$ and of~$G$;
indeed, Rosenlicht's lemma (see \cite[Lem.~10]{ctsansuctores}) shows that
\begin{align}
\bar k[Y]^*/\bar k^*=\bar k[\SL_n]^*/\bar k^* \oplus \bar k[T^{\alpha_1}]^*/\bar k^*
\oplus\dots
\oplus \bar k[T^{\alpha_m}]^*/\bar k^*\rlap{,}
\end{align}
while it is well known that $\bar k[\SL_n]^*=\bar k^*$ and that $\bar k[T^\alpha]^*/\bar k^*$,
for any $\alpha \in k^*$, is the
character group of the torus under which~$T^\alpha$ is a torsor.
On the other hand, by the exact sequence
\begin{align}
0\to \Z\to \Z[G]\to \Z[G]/\Z\to 0
\end{align}
and by the vanishing of $\HH^1(G,\Z)$, we have $\HH^0(G,\Z[G]/\Z)=0$.
We can now deduce from~\eqref{appendix:eqky} that $\bar k[X]^* = (\bar k[Y]^*)^G = \bar k^*$.
\end{proof}

Set $\widehat G=\Hom(G,\bar k^*)$.
We recall that the
\emph{type}
of the $G$\nobreakdash-torsor $\pi:Y \to X$
is, by definition, the isomorphism class of  $\pi \otimes_k \bar k:Y \otimes_k \bar k \to X \otimes_k \bar k$
as a $G$\nobreakdash-torsor
over $X\otimes_k \bar k$
and that it can be identified, thanks to Proposition~\ref{app:proposition}, with
a homomorphism $\lambda:\widehat G \to \Pic(X \otimes_k \bar k)$.
(See \cite[(3.3)]{HW18}, for this (standard) identification.)
The homomorphism~$\lambda$ is injective as~$G$ is finite and~$Y$ is geometrically connected
(see \cite[p.~40, Exercise~2]{Sko01}).

Let us denote by $\nu:\widehat T \hookrightarrow \Pic(X' \otimes_k \bar k)$
 the inverse image of $\lambda:\widehat G \hookrightarrow \Pic(X\otimes_k \bar k)$
by the restriction map $\Pic(X' \otimes_k \bar k)\to \Pic(X \otimes_k \bar k)$.
As in \cite[(3.1)]{HW18}, we have
a short exact sequence of $\Gal(\bar k/k)$\nobreakdash-modules
\begin{align}
\xymatrix{
0 \ar[r]& \widehat Q \ar[r]& \widehat T\ar[r]& \widehat G \ar[r]& 0\rlap{,}
}
\end{align}
where~$\widehat Q$ is a permutation $\Gal(\bar k/k)$\nobreakdash-module, and, dually, a short exact sequence
\begin{align}
\xymatrix{
1 \ar[r]&  G \ar[r]& T\ar[r]& Q \ar[r]& 1
}
\end{align}
of commutative algebraic groups over~$k$, where~$Q$ is a quasi-trivial torus and~$G$ is viewed as a constant
$k$\nobreakdash-group.  We note that~$T$ is a torus since $\Pic(X'\otimes_k\bar k)_{\mathrm{tors}}=0$
(see \cite[Prop.~1]{Ser59}).

As $X'(k)\neq\varnothing$, there exists a torsor over~$X'$, under~$T$, of type~$\nu$
(see \cite[Cor.~2.3.9]{Sko01}).
Applying \cite[Cor.~2.2]{HW18}\footnote{
All of the Brauer groups that appear in
Corollaire~2.2 of~\cite{HW18}
are unramified Brauer groups, hence this corollary is really a
 statement about Brauer--Manin sets of
smooth compactifications of  torsors,
even though
smooth compactifications do not figure explicitly in it.}
 to such a torsor,
we now see that
in order to prove Theorem~\ref{app:theorem},
it suffices to prove that rational points are dense in the Brauer--Manin set
for a smooth compactification of any torsor over~$X'$, under~$T$, of type~$\nu$.

\subsection{Fibration}
\label{app:subsecfib}

By \cite[Prop.~3.1]{HW18},
which we can apply
since
 $\bar k[X]^*=\bar k^*$
(see Proposition~\ref{app:proposition}),
any torsor over~$X'$, under~$T$, of type~$\nu$
contains an open subset~$W$ admitting a smooth
map $p:W \to Q$ whose fibres over the rational points of~$Q$ are torsors over~$X$, under~$G$, of type~$\lambda$.
In order to prove that rational points are dense in the Brauer--Manin set for a smooth
compactification of~$W$, we shall 
first prove that the base~$Q$ and
the fibres of~$p$ over the rational points of~$Q$ satisfy this property, then solve the
``fibration problem'' to deduce it for~$W$.

The variety~$Q$ is rational over~$k$ since it is a quasi-trivial torus, so the assertion on the base is trivial.
The fibre of~$p$ above any rational point of~$Q$
is in fact a twist~$Y^\sigma$ of~$Y$ by a $1$\nobreakdash-cocycle $\sigma \in Z^1(k,G)$,
since two torsors of a given type can only differ by such a twist.
As~$G$ acts diagonally
on $Y = \SL_n \times T^{\alpha_1} \times \dots \times T^{\alpha_m}$,
we have $Y^\sigma = (\SL_n)^\sigma \times (T^{\alpha_1})^\sigma \times \dots \times (T^{\alpha_m})^\sigma$.
On the one hand, we have $(\SL_n)^\sigma \simeq \SL_n$ since $\HH^1(k,\SL_n)$
is a singleton (Hilbert's Theorem~90);
hence $(\SL_n)^\sigma$ is rational over~$k$.
On the other hand, for any $\alpha \in k^*$,
the variety
$(T^\alpha)^\sigma$ is a torsor under the torus $(T^1)^\sigma$.
All in all~$Y^\sigma$ is birationally equivalent to a torsor under a torus over~$k$.
We conclude that for any smooth compactification~$Z$ of a fibre of~$p$ above a rational point of~$Q$,
the set $Z(k)$ is indeed dense in $Z(\A_k)^{\Br(Z)}$
(see
\cite[Thm.~6.3.1]{Sko01},
 \cite[Prop.~6.1~(iii)]{ctpalsko}).

A positive solution to the fibration problem for fibrations into rationally connected varieties
over a quasi-trivial torus is obtained
in \cite[Th.~4.2~(ii)]{HW18} under the assumption that a rational
section exists over~$\bar k$.
(The existence of such a rational section
ensures that the hypothesis of \emph{loc.\ cit.}\ is satisfied, as shown in
\cite[Lem.~1.1(b)]{Sko96}.)
Fortunately, this last condition holds in our situation.

\begin{proposition}
\label{app:propgenratpoint}
The generic fibre of $p\otimes_k \bar k:W \otimes_k \bar k \to Q \otimes_k \bar k$ possesses a rational point.
\end{proposition}

\begin{proof}
This generic fibre is a twist of $Y\otimes_k \bar k(Q)$ by a $1$\nobreakdash-cocycle
$\sigma \in Z^1(\bar k(Q),G)$.  Arguing as above, we see that it has a rational point
if and only if $(T^{\alpha_i} \otimes_k \bar k(Q))^\sigma$ has a rational point for each~$i$.
Writing $\alpha_i$ as a $|G|$\nobreakdash-th power in~$\bar k^*$
determines a $G$\nobreakdash-invariant $\bar k$\nobreakdash-point of $T^{\alpha_i}$,
hence a $G$\nobreakdash-equivariant
isomorphism $T^{\alpha_i} \otimes_k \bar k = T^1 \otimes_k \bar k$.
Thus $(T^{\alpha_i} \otimes_k \bar k(Q))^\sigma$ is isomorphic to $(T^1 \otimes_k \bar k(Q))^\sigma$,
a variety which certainly has a rational point since it is a torus.
\end{proof}

Applying
\cite[Th.~4.2~(ii)]{HW18} to a suitable compactification of~$p$
therefore completes the proof of Theorem~\ref{app:theorem}.

\subsection{Verticality of the Brauer group}
\label{app:subsecvertical}

It remains to prove Proposition~\ref{app:propvertical}.
As $X'(\A_k)^{\Br(X')}$ is closed in $X'(\A_k)$, we are free to replace~$x_v$, for~$v$ outside of an arbitrarily
large finite set of places of~$k$, with another $k_v$\nobreakdash-point of the same fibre of~$f$.
In particular, we may assume that $(x_v)_{v \in \Omega}$ is an adelic point of this fibre.
We may then view it as an adelic point of~$X$.
As such,
it is orthogonal,
for the Brauer--Manin
pairing $X(\A_k) \times \Br(X) \to \QQ/\ZZ$, to
$f^*\mathrm{Br}(B)$.
Proposition~\ref{app:propvertical} therefore results from the following purely algebraic statement,
in which~$k$ is allowed to be an arbitrary field of characteristic~$0$.

\begin{proposition}
\label{app:propverticalalg}
Viewing~$\Br(X')$ and $f^*\mathrm{Br}(B)$ as subgroups of~$\Br(X)$,
one has an inclusion $\Br(X') \subseteq f^*\mathrm{Br}(B)$.
\end{proposition}

\begin{proof}
Let~$V$ be a smooth compactification of the generic fibre~$V^0$ of~$f$.
As $V^0$
is a torsor under a torus over~$k(B)$
split by the extension $k(\SL_n)/k(B)$,
as $V^0\otimes_k\bar k$
is a torsor under a torus over $\bar k(B)$
split by the extension
$\bar k(\SL_n)/\bar k(B)$
 and as the natural map
 $\Gal(\bar k(\SL_n)/\bar k(B))\to \Gal(k(\SL_n)/k(B))$ is an isomorphism,
the following well-known lemma implies that
the pull-back map
$$
\Br(V) /  f^*\mathrm{\Br}(k(B)) \to
\Br(V\otimes_k\bar k) / f^*\mathrm{\Br}(\bar k(B))
$$
is injective.

\begin{lemma}
Let~$T$ be a torus over a field~$K$,
with character group~$\widehat T$,
split by
a finite Galois extension~$L/K$.
For any smooth and proper variety~$V$ over~$K$ containing a torsor under~$T$ as a dense open subset,
there is a canonical embedding
$\Coker(\Br(K)\to\Br(V))
\hookrightarrow \HH^2(\Gal(L/K),\widehat T)$.
\end{lemma}

\begin{proof}
Let~$\bar L$ denote a separable closure of~$L$
and $V^0$ the open subset in question.
As $\Br(V\otimes_K\bar L)=0$
and $\Br(L)\twoheadrightarrow \Br(V\otimes_K L)$,
the Hochschild--Serre spectral sequence provides an embedding
of
$\Coker(\Br(K)\to\Br(V))$
into the kernel of the restriction map
$\HH^1(K, \Pic(V \otimes_K\bar L))\to
\HH^1(L, \Pic(V \otimes_K\bar L))$,
that is, into $\HH^1(\Gal(L/K), \Pic(V\otimes_KL))$.
On the other hand,
the exact sequence
$$
0 \to \widehat T \to \Div_{(V \setminus V^0) \otimes_K L}(V \otimes_K L) \to \Pic(V \otimes_KL) \to 0
$$
(see \cite[p.~130]{Sko01}) embeds this group into $\HH^2(\Gal(L/K),\widehat T)$.
\end{proof}

As~$f$ is smooth and surjective,
we have
 $f^*\mathrm{\Br}(k(B)) \cap \Br(X') \subseteq f^*\mathrm{\Br}(B)$
as subgroups of $\Br(k(X))$.
It follows that the pull-back map
$$
\Br(X') / \big(\mathrm{Br}(X')\cap f^*\mathrm{\Br}(B)\big) \to
\Br(X'\otimes_k\bar k) / \big(\mathrm{Br}(X'\otimes_k\bar k)\cap f^*\mathrm{\Br}(B\otimes_k \bar k)\big)
$$
is injective as well.
Thanks to this injectivity,
we now see that in order to prove Proposition~\ref{app:propverticalalg}, we may assume that~$k$ is algebraically closed.

The generic fibre of the natural map
$X \to  (T^{\alpha_1} \times \dots \times T^{\alpha_m})/G$
is a left torsor under~$\SL_n$,
hence is isomorphic to~$\SL_n$ (Hilbert's Theorem~90).  It follows
that~$X$ is stably birationally equivalent to
$(T^{\alpha_1} \times \dots \times T^{\alpha_m})/G$.
This variety
is isomorphic to
$(T^1 \times \dots \times T^1)/G$
when~$k$ is algebraically closed,
as we have seen in the proof of Proposition~\ref{app:propgenratpoint}.
In addition,
the unramified
Brauer group of
$(T^1 \times \dots \times T^1)/G$
vanishes
when~$k$ is algebraically closed,
by Saltman's formula
\cite[Thm.~8.7]{ctsrationality} and by the next lemma.
Hence $\Br(X')=0$ in this case.

\begin{lemma}\label{lem:bog-multiplier}
Let $\mathcal B_G$ denote the set of
subgroups of~$G$ generated by two elements.
For any finite abelian group~$G$
and for $M=\Z[G]/\Z$ or $M=\QQ/\ZZ$,
the product of restriction maps
$\HH^2(G, M) \to \prod_{H \in \mathcal B_G} \HH^2(H, M)$
is injective.
\end{lemma}

\begin{proof}
As $\HH^2(G,\QQ)$ and $\HH^2(G,\ZZ[G])$ vanish, this follows from the injectivity
of the product of restriction maps
\begin{align}
\label{app:restrmap}
\HH^3(G,\ZZ) \to \prod_{H \in \mathcal B_G} \HH^3(H,\ZZ)\rlap.
\end{align}
It is a general fact, valid for an arbitrary finite group~$G$,
that the kernel of~\eqref{app:restrmap} remains unchanged
if one replaces~$\mathcal B_G$ with the set of abelian subgroups
of~$G$ (see
\cite[Thm.~7.1]{ctsrationality}), which implies the desired injectivity when~$G$ is abelian.
Alternatively, this injectivity
results from Lemma \ref{lem:G_i}  and \cite[Lem.~6.4]{HNP}.
\end{proof}

This completes the proof of Proposition~\ref{app:propverticalalg}.
\end{proof}

\subsection{Nonabelian Galois groups}\label{app:nonabelien}
\newcommand{\ab}{\mathrm{ab}}
\newcommand{\nr}{\mathrm{nr}}
\def\Beta{\mathfrak{P}}
\def\<{\langle}
\def\>{\rangle}
The descent-fibration argument described in \S\ref{app:subsecdescent} and \S\ref{app:subsecfib} is modelled after a similar argument appearing in~\cite{HW18}, profiting in addition from the favourable circumstance of $G$ being abelian. In general, the inductive argument of~\cite{HW18} is constructed to handle also nonabelian groups, as long as they admit a suitable filtration into normal subgroups whose successive quotients are cyclic; such groups are also known as \emph{supersolvable}. Though the variety $X$ considered here is more complicated than the one considered in~\cite{HW18}, 
the argument of \emph{loc.\ cit.}\ can be adapted to yield the statement of Theorem~\ref{thm:existence} for any supersolvable $G$, see~\cite{supersolvable}. 
Interestingly enough, though, it turns out that the stronger claim appearing in Corollary~\ref{cor:existence} \emph{does not} hold for a general nonabelian group $G$,
even when~$G$ is supersolvable (indeed, even when~$G$ is a $2$\nobreakdash-group). This is due to the fact that the variety $X$ may contain unramified Brauer classes which are not vertical with respect to the projection $f:X \to B$, and which can obstruct the weak approximation of local points on $X$, even when those local points lie over a rational point of~$B$. (Such Brauer classes do not exist in the abelian case; see Proposition~\ref{app:propverticalalg}.)
Let us now illustrate how one can construct a nonabelian example where exactly this happens.

We shall say that a group $H$ is \emph{weakly bicyclic} if it is an extension of a cyclic group by
a cyclic group. We note that if $K/k$ is a Galois extension with Galois group $G$ then the decomposition subgroups $H_v \subseteq G$ are weakly bicyclic at every finite place $v$ which does not divide the order of $G$. Given a group $G$, we shall denote by $\mathcal B_G$ the set of weakly bicyclic subgroups of~$G$
(a notation compatible with Lemma~\ref{lem:bog-multiplier} when~$G$ is abelian).

\begin{proposition}\label{p:example}
Let $G$ be a finite $2$-group satisfying the following properties:
\begin{enumerate}
\item[(i)] $G$ has exponent $\leq 16$.
\item[(ii)]
The abelianization $G^{\ab}$ has exponent $2$ and is generated by images of elements of $G$ of order $2$.
\item[(iii)]
There exists an element $\varphi \in \HH^2(G,\ZZ/2\ZZ)$
whose restriction to every cyclic subgroup of~$G$ of order $16$ vanishes, whose restriction to at least one cyclic subgroup of~$G$ of order $8$
does not vanish,
and
whose image by the natural map $\delta:\HH^2(G,\ZZ/2\ZZ) \to \HH^2(G, \QQ/\ZZ)$ belongs to, and spans, the kernel of the product of restriction maps
$\HH^2(G, \QQ/\ZZ) \to \prod_{H \in \mathcal B_G} \HH^2(H, \QQ/\ZZ)$.
\end{enumerate}
Let $H \subseteq G$ be a cyclic subgroup of order $8$ on which $\varphi$ does not vanish.  Then:
\begin{enumerate}
\item[(1)]\label{i:1}
There exist $G$-extensions $K/\QQ$ which are unramified at $2$ and whose decomposition groups at $2$ are conjugate to $H$. 
\item[(2)]
For every $G$-extension $K/\QQ$ as in~\hyperref[i:1]{(1)}, the element $256 \in \QQ^*$ is a local norm from $K$ at every place of $\QQ$, but not a global norm from $K$.
\end{enumerate}
In particular, the statement of Corollary~\ref{cor:existence} does not hold for $G$ with $k=\QQ$, $S=\{2\}$ and $\mathcal{A} \subset k^*$ the subgroup generated by $256$. 
\end{proposition}

The proof of Proposition~\ref{p:example} requires a bit of preparation.
In the next lemma, we denote by $\Br_{\nr}(B)$, $\Br_1(B)$, $\Br_{1,\nr}(B)$, $\Br_0(B)$ the subgroups of~$\Br(B)$
consisting, respectively, of unramified, algebraic, algebraic unramified,  constant classes.

\begin{lemma}\label{l:example}
Let $G \subseteq \SL_n(\QQ)$ be a finite subgroup. Let $B=\SL_n/G$.
\begin{enumerate}
\item
If $G$ satisfies Condition~(ii) of Proposition~\ref{p:example}, then $\Br_{1,\nr}(B) = \Br_0(B)$.
\item
If $G$ satisfies Condition~(iii) of Proposition~\ref{p:example}, then $\Br_{\nr}(B)=\Br_{1,\nr}(B)$. 
\end{enumerate}
In particular, for $G$ as in Proposition~\ref{p:example}, we have 
$\Br_{\nr}(B)=\Br_0(B)$.
\end{lemma}
\begin{proof}
Condition (ii) implies that for any field $K$, the group $\HH^1(K,G^{\ab})$ is generated by elements in the image of the pointed set $\HH^1(K,G)$ (and even by elements  coming from $\HH^1(K,\ZZ/2\ZZ)$ via 
homomorphisms $\Z/2\Z\to G$). The first claim then follows, by local and global duality, from
 \cite[Prop.~4]{Harari}. 
Let us now explain why Condition~(iii) implies that $\Br_{\nr}(B)=\Br_{1,\nr}(B)$. 
For every subgroup $H \subseteq G$, the Hochschild--Serre spectral sequences for the $H$-coverings $\pi_H:\SL_n \to \SL_n/H$ and $\SL_{n,\bar \QQ} \to \SL_{n,\bar \QQ}/H$,
together with the inclusion of roots of unity $\mu_{\infty} \subseteq \bar \QQ^*$,
give rise to a commutative diagram
\begin{align}
\begin{aligned}
\label{app:diag:hs}
\xymatrix@R=2.5ex@C=1.75em{
\Ker\bigl(\Br(\SL_n/H) \to \Br(\SL_n)\bigr) \ar[d] & \HH^2(H,\QQ^*) \ar[l]_(.3)\sim \ar[d] & \HH^2(H,\mu_2) \ar[l] \ar@{->>}[d] \\
\Ker\bigl(\Br(\SL_{n,\bar \QQ}/H) \to \Br(\SL_{n,\bar \QQ})\bigr)^{\Gamma_\QQ} & \HH^2(H,\bar \QQ^*)^{\Gamma_\QQ} \ar[l]_(.3)\sim & \HH^2(H,\mu_{\infty})^{\Gamma_\QQ} \ar[l]_{\sim}\rlap{,}
}
\end{aligned}
\end{align}
where $\Gamma_\QQ=\Gal(\bar\QQ/\QQ)$ is the absolute Galois group of $\QQ$. The horizontal arrows between the first two columns are isomorphisms since $\Pic(\SL_{n}) = \Pic(\SL_{n,\bar\QQ}) = 0$,
and the bottom right horizontal map is an isomorphism since $\bar \QQ^*/\mu_{\infty}$ is uniquely divisible. In addition, the rightmost vertical map is surjective: indeed, this map fits in the middle of the commutative diagram with exact rows
\[ \xymatrix{
0 \ar[r] & \Ext^1(\HH_1(H),\mu_2) \ar[r]\ar[d] & \HH^2(H,\mu_2) \ar@{->>}[r]\ar@{->>}[d] & \Hom(\HH_2(H),\mu_2)
\ar[d]^(.45){\wr}  \\
0 \ar[r]^-{\sim} & \Ext^1(\HH_1(H),\mu_\infty)^{\Gamma_\QQ} \ar[r] & \HH^2(H,\mu_\infty)^{\Gamma_\QQ} \ar[r]^-{\sim} & \Hom(\HH_2(H),\mu_\infty)^{\Gamma_\QQ} 
}\]
determined by the universal coefficient theorem, where $\Ext^1(\HH_1(H),\mu_\infty) = 0$ since $\mu_\infty$ is a divisible group. 

We now fix a $\beta \in \Br_{\nr}(B)$ and aim to show that $\beta$ is algebraic. By adding to $\beta$ a constant class, we may assume that $\beta(\pi_G(1))=0$.
As~$\SL_n$ is rational over~$\QQ$, we have $\Br_\nr(\SL_n)=\Br_0(\SL_n)$,
and so $\pi_G^*\mkern.5mu\beta=0$.
Considering the diagram~\eqref{app:diag:hs} for $G=H$ and using the surjectivity
of its right vertical map, we find $\beta_G \in \HH^2(G,\mu_2)$ whose eventual image in $\Br(B_{\bar \QQ})$ 
is the same as the image of $\beta$. 
Now by Bogomolov's formula (see, e.g., \cite[Thm.~7.1]{ctsrationality}),
the group $\Br_{\nr}(\SL_{n,\bar \QQ}/H)$ vanishes whenever~$H$ is weakly bicyclic, and so
by the naturality of~\eqref{app:diag:hs},
the image of~$\beta_G$ in $\HH^2(H,\mu_{\infty})$ 
vanishes for every $H \in \mathcal B_G$.
Since $\mu_{\infty} \cong \QQ/\ZZ$ as abelian groups via a choice of a compatible system of roots of unity, Condition~(iii) implies that the image of $\beta_G$ in $\HH^2(G,\mu_\infty)$ is either~$0$ or the image of $\varphi \in \HH^2(G,\ZZ/2) = \HH^2(G,\mu_2)$ under the natural map $\HH^2(G,\mu_2) \to \HH^2(G,\mu_{\infty})$.
By possibly amending the choice of~$\beta_G$,
we may assume that
 $\beta_G \in \{0,\varphi\}$. We then write $\beta_1 \in \Br(B)$ for the image of $\beta_G$, and set $\beta_2 := \beta - \beta_1$. By construction, $\beta_1$ (and hence also $\beta_2$) vanishes when pulled back to $\SL_n$, and $\beta_2$ also vanishes when pulled back to $B_{\bar \QQ}$. In particular, $\beta_2 \in \Ker(\Br_1(B) \to \Br_1(\SL_n))$. 

Let now $H \subseteq G$ be a cyclic subgroup of order~$8$ on which~$\varphi$ does not vanish.
As~$\beta$ is unramified,
there exists a prime~$p_0$ such that~$\beta$
evaluates trivially on $B(\QQ_p)$ for all $p>p_0$.
Choose~$p>p_0$ such that
there exists a cyclic extension $L/\QQ_p$ of degree~$4$ that does not extend to a cyclic extension of degree~$8$
(any~$p$ such that~$-1$ is a square but not a $4$th power modulo~$p$ will do).
Embed $\Gal(L/\QQ_p)$ into~$H$.
The image of the class of $L/\QQ_p$ by the resulting
map $\HH^1(\QQ_p,\Gal(L/\QQ_p)) \to \HH^1(\QQ_p,G)$
is the class of the torsor
$\pi_G^{-1}(b)$ for some point
 $b \in B(\QQ_p)$ (see \cite[\textsection1.2]{Harari}), which we fix. 
 
By the choice of $p$, we have $\beta(b) = 0$. We claim that $\beta_2(b)=0$ as well.
Indeed, as $\beta_2$ is algebraic and $\beta_2(\pi_G(1)) = 0$, it follows that $\beta_2$ vanishes when
pulled back to the universal torsor $\SL_n/G'\to B$, where $G' := \Ker(G \to G^{\ab})$ is the derived subgroup of $G$. On the other hand, since~$G^\ab$ has exponent~$2$ and the subgroup $\Gal(L/\QQ_p) \subseteq H$ consists of elements divisible by~$2$, we have $\Gal(L/\QQ_p) \subseteq G'$. This means that $b$ lifts to $\SL_n/G'$ and so $\beta_2(b)=0$. Therefore $\beta_1(b)=0$.

Let us prove that $\beta_G=0$.
By contradiction, assume that $\beta_G = \varphi$. Then the restriction of $\beta_G$ to~$H$ is the non-trivial element of $\HH^2(H,\ZZ/2\ZZ) = \ZZ/2\ZZ$, which is the one classifying the central extension $\tilde{H} \to H$ with $\tilde{H}$ cyclic of order 16. This element restricts to the non-trivial element of $\HH^2(\Gal(L/\QQ_p),\ZZ/2\ZZ)$ for the cyclic order 4 subgroup $\Gal(L/\QQ_p) \subseteq H$, and further to a non-trivial element of $\HH^2(\QQ_p,\ZZ/2\ZZ)$ by the assumption that $L$ does not extend to a cyclic degree 8 extension. 
As $\beta_1(b)= 0 $, this is absurd.
We conclude that $\beta_G = 0$ and $\beta = \beta_2$,
which completes the proof that $\Br_{\nr}(B)=\Br_{1,\nr}(B)$.
\end{proof}

\begin{proof}[Proof of Proposition~\ref{p:example}]
Fix an embedding $G \hookrightarrow \SL_n(\QQ)$ 
and let $B=\SL_n/G$. By Lemma~\ref{l:example}, we have $\Br_{\nr}(B) = \Br_0(G)$, and so the existence of a $G$\nobreakdash-extension $K/\QQ$ as in~\hyperref[i:1]{(1)} follows from~\cite[Th.~B]{HW18},
since $G$ is nilpotent and in particular supersolvable. Let us  choose such an extension $K/\QQ$. As $256$ is positive, it is a norm from $K_{\infty}$. In addition, $16$ is an $8$th power
(and hence $256$ is a $16$th power)
 in~$\QQ_p$ for every odd~$p$; indeed, one of $2,-2$ or $-1$ is a square, and in the latter case $2i=(1+i)^2$ is a square.
Since $G$ has exponent at most $16$ and $256$ is a unit outside $2$, it follows that in~$K/\QQ$,
the element $256$ is a local norm at every odd finite place. Finally, at $2$ the extension $K/\QQ$ is unramified with Frobenius element of order $8$, and hence $256=2^8$ is a norm from $K_2$ as well.

It is left to show that $256$ is not a global norm from $K$.  Let us first recall some notation used above.
For $\alpha \in k^*$, we consider the variety $X^{\alpha} := (\SL_n \times T^{\alpha})/G$, equipped with the projection $X^{\alpha} \to B$.
 Given a rational point $b \in B(k)$, the fibre of $X^1 \to B$ over $b$ is naturally isomorphic to the norm~$1$ torus of $K/k$, which we denote by $T^1_b$. 
The fibre $X^{\alpha}_b$ of $X^{\alpha} \to B$ over $b$ is then naturally isomorphic to the norm $\alpha$ torsor of $T^1_b$. We also recall that $\widehat{T}^1$ denotes the character lattice of $T^1$, which carries a natural action of $G$.

Let $b \in B(\QQ)$ be such that $[b] \in  \HH^1(\QQ,G)$ classifies the extension $K/\QQ$ (see \cite[\textsection1.2]{Harari}).
In order to prove that $X^{256}_b(\QQ) = \varnothing$, i.e.\ that $256$ is not a norm from~$K$,
we shall now exhibit a Brauer--Manin obstruction on $X^{256}_b$.

As $G=\pi_1(B,\pi_G(1))$,
we may identify  $G$\nobreakdash-modules with locally constant \'etale sheaves of abelian
groups on~$B$.
In this way, we view $\widehat T^1/2\widehat{T}^1$ as an \'etale sheaf on~$B$
and~$\varphi$
as an element
of $\HH^1(B,\widehat T^1/2\widehat{T}^1)$
via the natural isomorphism $\HH^1(B,\widehat{T}^1/2\widehat{T}^1)= \HH^2(G,\ZZ/2\ZZ)$.
The map $p:X^{16} \to X^{256}$ induced by the squaring map $T^{16} \to T^{256}$
is a torsor under the
 $2$\nobreakdash-torsion subgroup scheme of $X^1\to B$,
which we identify with the $G$\nobreakdash-module $T^1[2]=\{x \in T^1(\bar\QQ)\mkern2mu:\mkern1.5mu x^2=1\}$;
this torsor is classified by an element $\psi \in  \HH^1(X^{256},T^1[2])$.
We set
$ \Beta := \psi \cup \varphi \in  \HH^2(X^{256},\mu_2)$.
We will abusively identify $\Beta$ with its image in $\Br(X^{256})$ and consider it as a \emph{Brauer element} of order $2$.

We have already seen that $X^{256}_b(\QQ_v)\neq\varnothing$ for any place~$v$ of~$\QQ$.
Let us show that the evaluation
$$
\sum_v \inv_v (x_v^*\Beta) \in \QQ/\ZZ
$$
is well defined and non-zero for any collection $(x_v)_v$ of local points in $X^{256}_b$.

Our assumptions on $\varphi$ imply that we can choose, for every weakly bicyclic subgroup $H \subseteq G$, a class $\widetilde{\varphi}_H \in \HH^1(H,\QQ/\ZZ)$ whose image,
under the boundary map $\HH^1(H,\QQ/\ZZ) \to \HH^2(H,\ZZ/2\ZZ)$,
is the restriction $\varphi_{H} \in \HH^2(H,\ZZ/2\ZZ)$ of $\varphi$.

For any place~$v$ of~$\QQ$,
let~$K_v$
denote the completion of~$K$ at a place of~$K$ dividing~$v$.
The corresponding decomposition group $D_v \subseteq G$
is weakly bicyclic since~$G$ is a $2$\nobreakdash-group and $K_2/\QQ_2$ is unramified.
Letting $\QQ_v \subseteq K_{\widetilde{\varphi}_{D_v}} \subseteq K_v$ denote the intermediate cyclic extension determined by $\widetilde{\varphi}_{D_v} \in \HH^1(D_v,\QQ/\ZZ)$,
a direct computation now reveals that
$x_v^*\Beta = (16,K_{\widetilde{\varphi}_{D_v}}/\QQ_v) \in \Br(\QQ_v)$.

Since $\varphi$ is assumed to vanish on every cyclic subgroup of order $16$, the class $\widetilde{\varphi}_{D_v}$ becomes divisible by $2$ when restricted to every such subgroup. Since the exponent of~$G$ divides~$16$, it follows that
 $8\widetilde{\varphi}_{D_v} \in \HH^1(D_v,\QQ/\ZZ) = \Hom(D_v,\QQ/\ZZ)$
vanishes when restricted to any cyclic subgroup of~$D_v$, and hence vanishes;
in other words,
the degree of the extension $K_{\widetilde{\varphi}_{D_v}}/\QQ_v$
divides~$8$. On the other hand, as~$D_2$ is cyclic of order~$8$ and $\varphi$ does not vanish when restricted to~$D_2$ we have that $\widetilde{\varphi}_{D_2} \in \HH^1(D_2,\QQ/\ZZ) \cong \ZZ/8\ZZ$ is not divisible by $2$ and so
$K_{\widetilde{\varphi}_{D_2}} = K_2$.
We conclude that $\inv_v (x_v^*\Beta) = 0$
for all $v \neq 2$ 
(recall that $16$ is an $8$th power at such~$v$)
while $\inv_2 (x_2^*\Beta) = 1/2 \in \QQ/\ZZ$.
\end{proof}

We shall now construct a $2$-group $G$ satisfying the conditions of Proposition~\ref{p:example}.
Let $N$ be the group generated by $4$ generators $x,y,z_+,z_-$ under the following relations:
\begin{enumerate}
\item
$x^{16} = y^{16} = z_+^8 = z_-^8 = 1$;
\item
each of $z_+,z_-$ commutes with each of $x,y,z_+,z_-$;
\item
$[x,y] = z_+z_-$. 
\end{enumerate}
In particular, $N$ is a central extension of the bicyclic group $\ZZ/16\ZZ\<x,y\>$ by the bicyclic group
$\ZZ/8\ZZ\<z_+,z_-\>$.
Let $\sigma: N \to N$ be the involution given by $\sigma(x)=x^{-1},\sigma(y)=y^{-1}, \sigma(z_+)=z_-$ and $\sigma(z_-)=z_+$. We define $G := N \rtimes \ZZ/2\ZZ\<\sigma\>$ to be the associated semi-direct product and view~$\sigma$ as an element of~$G$.

It is straightforward that $G$ satisfies Conditions (i) and (ii) of Proposition~\ref{p:example}.
Let us now construct an element $\varphi \in  \HH^2(G,\ZZ/2\ZZ)$ satisfying Condition~(iii).
The homomorphism
 $\rho:N \to \ZZ/8\ZZ$
which sends $x,y$ to $0$, $z_+$ to $1$ and $z_-$ to $-1$ intertwines the action of $\sigma$ with the action of $-1: \ZZ/8\ZZ \to \ZZ/8\ZZ$. Consequently, it induces a homomorphism
$\rho':G = N \rtimes \ZZ/2\ZZ \to \ZZ/8\ZZ \rtimes \ZZ/2\ZZ =: D_8$ 
to the dihedral group of order $16$. Consider the short exact sequence
\begin{equation}\label{e:dih} 
1 \to \ZZ/2\ZZ \to D_{16} \stackrel{q}{\to} D_8 \to 1 
\end{equation}
where $D_{16} := \ZZ/16\ZZ \rtimes \ZZ/2\ZZ$ is the dihedral group of order $32$ and the map $q$ is induced by the surjective map $\ZZ/16\ZZ \to \ZZ/8\ZZ$. Let $\varphi_{D_8} \in  \HH^2(D_8,\ZZ/2\ZZ)$ be the element classifying the central extension~\eqref{e:dih} and let $\varphi := (\rho')^*\varphi_{D_8} \in  \HH^2(G,\ZZ/2\ZZ)$.
We leave it to the reader to verify that $\varphi$ has the desired properties.

\end{document}